\documentclass[12pt,reqno]{amsart}
\usepackage{pgf,tikz}
\usepackage{subfigure}
\usetikzlibrary{arrows}
\usepackage{amsmath}
\usepackage{amssymb}
\usepackage{amsthm}
\usepackage{amscd}

\usepackage{delarray}
\usepackage{hyperref}
\headheight=14.4pt \headsep=1cm \numberwithin{equation}{section}
\def\P{{\mathbb P}}

\def\Z{{\mathbb Z}}

\newtheorem{theorem}{Theorem}[section]
\newtheorem{lemma}[theorem]{Lemma}
\newtheorem{proposition}[theorem]{Proposition}
\newtheorem{corollary}[theorem]{Corollary}

\theoremstyle{definition}

\newtheorem{remark}[theorem]{Remark}

\newtheorem{convention and reminder}[theorem]{Convention and Reminder}
\newtheorem{convention and remark}[theorem]{Convention and Remark}
\newtheorem{definition and remark}[theorem]{Definition and Remark}

\newtheorem{reminders and definition}[theorem]{Reminders and Definition}

\newtheorem{notation and remarks}[theorem]{Notation and Remarks}
\newtheorem{notation and remark}[theorem]{Notation and Remark}
\newtheorem{notation and reminder}[theorem]{Notation and Reminder}
\newtheorem{example}[theorem]{Example}

\newcommand\Ker{\operatorname{\Ker}}

\title[On Curves lying on a rational normal surface scroll]{On Curves lying on a rational normal surface scroll}

\author{Wanseok Lee}
\address{Department of Applied Mathematics, Pukyong National University, Busan 608-737, Korea}
\email{wslee@pknu.ac.kr}

\author{Euisung Park}
\address{Department of Mathematics, Korea University, Seoul 136-701, Korea}
\email{euisungpark@korea.ac.kr}

\begin{document}

\date{Busan, Seoul, \today}


\maketitle

\begin{abstract}
In this paper, we study the minimal free resolution of non-ACM
divisors $X$ of a smooth rational normal surface scroll $S=S(a_1
,a_2 ) \subset \P^r$. Our main result shows that for $a_2 \geq 2a_1
-1$, there exists a nice decomposition of the Betti table of $X$ as
a sum of much simpler Betti tables. As a by-product of our results,
we obtain a complete description of the graded Betti numbers of $X$
for the cases where $S=S(1,r-2)$ for all $r \geq 3$ and
$S=S(2,r-3)$ for all $r \geq 6$.
\end{abstract}

\section{Introduction}
\noindent Let $X \subset \P^r$ be a nondegenerate projective
subvariety defined over an algebraically closed field $\Bbbk$.
Various interesting properties of $X$ can be obtained from the
minimal graded free resolution of its homogeneous vanishing ideal.
But there are only a few cases where the free resolution is
completely known.

The purpose of this paper is to study the minimal free resolution of
$X$ when it is a curve lying on a smooth rational normal surface
scroll.

Recall that if $S \subset \P^r$ is a nondegenerate projective
surface then its degree is at least $r-1$, and $S$ is called a
\textit{surface of minimal degree} when ${\rm deg}(S) = r-1$. It is
well-known that $S$ is either a quadric of rank $=4$ or the Veronese
surface in $\P^5$ or a rational normal surface scroll (cf. \cite{EH}). There have been several results which show that projective curves that are contained in a surface of minimal degree behave extremally
with respect to various properties. More precisely, let $X \subset
\P^r$ be a nondegenerate projective integral curve of degree $d$.
Let $R := \Bbbk [x_0 , \ldots , x_r ]$ be the homogeneous coordinate
ring of $\P^r$ and $I(X)$ the defining ideal of $X$. The graded
Betti numbers of $X$ are defined by
\begin{equation*}
\beta_{i,j} (X) := {\rm dim}_{\Bbbk} {\rm Tor}^R _i (I(X) ,
\Bbbk)_{i+j}
\end{equation*}
and the Betti table of $X$, denoted by $\beta (X)$, is the table
whose entry in the $i$-th column and $j$-th row is $\beta_{i,j}(X)$.
Throughout this paper, we present $\beta (X)$ as follows:\\

\begin{center}
$\beta (X)=$
\begin{tabular}{|c||c|c|c|c|c|c|c|}\hline
$i$           & $0$               & $1$               & $\cdots$ & $i$               & $\cdots$ & $r-1$               & $r$                \\\hline
$\vdots$      & $\vdots$          & $\vdots$          & $\ddots$ & $\vdots$          & $\ddots$ & $\vdots$            & $\vdots$            \\\hline
$\beta_{i,3}$ & $\beta_{0,3} (X)$ & $\beta_{1,3} (X)$ & $\cdots$ & $\beta_{i,3} (X)$ & $\cdots$ & $\beta_{r-1,3} (X)$ & $\beta_{r,3} (X)$  \\\hline
$\beta_{i,2}$ & $\beta_{0,2} (X)$ & $\beta_{1,2} (X)$ & $\cdots$ & $\beta_{i,2} (X)$ & $\cdots$ & $\beta_{r-1,2} (X)$ & $\beta_{r,2} (X)$  \\\hline
\end{tabular}
\end{center}
\smallskip

\noindent For example, $\beta_{0,2} (X)$ is the number of quadratic
generators of $I(X)$.

The classical Castelnuovo Lemma shows that if $d \geq 2r+1$, then $\beta_{0,2} (X) \leq {{r-1} \choose {2}}$ and
equality is attained if and only if $X$ lies on a surface of minimal
degree. G. Castelnuovo gave an upper bound of the arithmetic genus
of $X$ and proved that his bound is achieved only if $X$ lies on a
surface of minimal degree (cf. \cite{H}). M. Green's $K_{p,1}$
Theorem in \cite{G} says that $\beta_{i,2} (X) =0$ if $i \geq r-1$,
$\beta_{r-2,2} (X) \neq 0$ if and only if $X$ is a rational normal
curve, and $\beta_{r-3,2} (X) \neq 0$ if and only if $X$ lies on a
surface of minimal degree. Also, it is proved in \cite{MV} and
\cite{M} that if $X$ is a $k$-Buchsbaum curve then
\begin{equation*}
{\rm reg}(X) \leq \left\lceil \frac{d-1}{r-1} \right\rceil + {\rm
max} \{ k,1 \},
\end{equation*}
and when $k >0$ and $d \geq 2r^2 -3r+3$, the equality ${\rm reg}(X)
= \left\lceil \frac{d-1}{r-1} \right\rceil + k$ holds only if $X$
lies on a surface of minimal degree.

The above results lead our attention to the problem of studying the
minimal free resolution of $X$ when it is a curve contained in a
surface $S$ of minimal degree. In \cite[Theorem 2.4]{N}, U. Nagel
obtains a complete description of $\beta (X)$ when X is
arithmetically Cohen-Macaulay. Note that $X$ is always ACM if $S$ is
the Veronese surface in $\P^5$ or a singular rational normal surface
scroll (cf. \cite[Proposition 2.9]{N} and \cite[Example 5.2]{Fe}).

Now, let $S=S(a_1 , a_2 )$ be a smooth rational normal scroll in
$\P^r$ such that $1 \leq a_1 \leq a_2$ and $r=a_1 +a_2 +1$. Thus the
divisor class group of $S$ is freely generated by the hyperplane
section $H$ and a ruling line $F$ of $S$. When $X$ is linearly
equivalent to $aH+bF$, it is non-degenerate in $\P^r$ if and only if
\begin{equation}\label{eq:1.1}
\mbox{either} \quad  a=0 ~ \mbox{and} ~ b>a_2 \quad \mbox{or} \quad
a=1 ~ \mbox{and} ~ b \geq 1 \quad \mbox{or} \quad a \geq 2 ~
\mbox{and} ~ b \geq -aa_2
\end{equation}
(cf. \cite[Lemma 2.2]{P2}). Concerned with the minimal free
resolution of $X$, it is an interesting and important property that
$\beta(X)$ is invariant inside the divisor class of $X$. That is, if
$X'$ is a curve in $S$ and $X' \equiv X$, then $\beta (X) = \beta
(X' )$ (cf. \cite[Proposition 4.1]{P2}). Finally, note that the
graded Betti numbers of $X$ are completely known when $a \geq 1$ and
$b \leq 1$ (cf. \cite[Theorem 4.3 and Theorem 4.4]{P2}).
Along this line, a more precise goal of this paper is to study the following problem.\\

\begin{enumerate}
\item[] {\bf Problem ($\dagger$)}. Let $S$ and $X$ be as above such that
\begin{equation*}
{\rm either} \quad a=0 ~ \mbox{and} ~ b>a_2 \quad \mbox{or} \quad a
\geq 1 ~ \mbox{and} ~ b \geq 2.
\end{equation*}
Then describe $\beta (X)$ completely (in terms of the integers $a_1$, $a_2$, $a$ and $b$). \\
\end{enumerate}

The first general result associated with this problem is

\begin{theorem}[Theorem 4.8 in \cite{G-M}]\label{thm:GM}
Let $X$ be an effective divisor of the smooth quadric
$S=S(1,1)$ in $\P^3$ which is linearly equivalent to $aH+bF$ where
$a \geq 0$ and $b \geq 2$. Then  \\

\begin{center}
$\beta (X) =$
\begin{tabular}{|c||c|c|c|c|}\hline
$i$ & $0$ & $1$ & $2$  \\\hline   $\beta_{i,a+b}$ & $b+1$ & $2b$ &
$b-1$      \\\hline $\beta_{i,a+b-1}$  & $0$ & $0$ & $0$
\\\hline $\vdots$ & $\vdots$ & $\vdots$ & $\vdots$
\\\hline $\beta_{i,3}$  & $0$ & $0$  & $0$                \\\hline
$\beta_{i,2}$  & $1$ & $0$  & $0$
\\\hline
\end{tabular}.
\end{center}
\end{theorem}
\smallskip

In this paper, we extend Theorem \ref{thm:GM} to all $S$ and $X$ that satisfy the following conditions (\ref{eq:1.2}) :
\begin{equation}\label{eq:1.2}
a_2 \geq 2 a_1 -1 \quad \mbox{and} \quad b \equiv \epsilon(\mod~a_2
) \quad \mbox{for some} \quad a_1 +1 \leq \epsilon \leq a_2 +1
\end{equation}
To be more precise, Theorem \ref{Thm:Decomposition 2} says that if
$S$ and $X$ satisfy (\ref{eq:1.2}), then $\beta (X)$ is expressed as
the sum of several Betti tables that are much simpler. Also in
Propositions \ref{prop:Betti number} and its corollaries, we obtain
a complete description of those simpler Betti diagrams that make up
$\beta (X)$.

When $a_1 =1$ and hence $a_2 = r-2$ for  all $r \geq 3$, every $X$
considered in Problem ($\dagger$) satisfies the conditions in
(\ref{eq:1.2}). So, using Theorem \ref{Thm:Decomposition 2} and
Propositions \ref{prop:Betti number}, we solve Problem ($\dagger$)
entirely in case of $a_1 =1$. For details, see Theorem
\ref{Thm:a1=1}. In particular, our results reprove Theorem
\ref{thm:GM}. The following two theorems are obtained by applying
Theorem \ref{Thm:a1=1} to the cases where $S=S(1,2)$ and $S=S(1,3)$,
respectively. These results and their proofs illustrate how the main
results of this paper can be applied in specific cases.

\begin{theorem}\label{thm:S(1,2)}
Let $X$ be an effective divisor of $S=S(1,2)$ in $\P^4$ linearly
equivalent to $aH+bF$ where either $a=0$ and $b \geq 3$ or else $a
\geq 1$ and $b \geq 2$. Then $\beta (X)$ is equal to the first
(resp. the second) one of the following two tables in Table 1 when
$b=2 \delta$ (resp. $b=2 \delta +1$):

\begin{table}[hbt]\label{table:S(1,2)}
\begin{center}
\begin{tabular}{|c||c|c|c|c|c|}\hline
$\beta_{i,a+b}$ & $1$      & $3$      & $3$      & $1$      \\\hline
$\vdots$               & $\vdots$ & $\vdots$ & $\vdots$ & $\vdots$
\\\hline $\beta_{i,a+\delta+2}$ & $1$      & $3$      & $3$      &
$1$      \\\hline $\beta_{i,a+\delta+1}$ & $1$      & $6$      & $5$
& $1$      \\\hline $\beta_{i,a+\delta}$   & $1$      & $0$      &
$0$      & $0$      \\\hline $\vdots$               & $\vdots$ &
$\vdots$ & $\vdots$ & $\vdots$  \\\hline $\beta_{i,3}$          &
$0$      & $0$      & $0$      & $0$       \\\hline $\beta_{i,2}$ &
$3$      & $2$      & $0$      & $0$       \\\hline
\end{tabular} $\quad$
\begin{tabular}{|c||c|c|c|c|c|}\hline
$\beta_{i,a+b}$ & $1$      & $3$      & $3$      & $1$      \\\hline
$\vdots$               & $\vdots$ & $\vdots$ & $\vdots$ & $\vdots$
\\\hline $\beta_{i,a+\delta+2}$ & $1$      & $3$      & $3$      &
$1$      \\\hline $\beta_{i,a+\delta+1}$ & $3$      & $6$      & $3$
& $0$      \\\hline $\beta_{i,a+\delta}$   & $0$      & $0$      &
$0$      & $0$       \\\hline $\vdots$               & $\vdots$ &
$\vdots$ & $\vdots$ & $\vdots$  \\\hline $\beta_{i,3}$          &
$0$      & $0$      & $0$      & $0$       \\\hline $\beta_{i,2}$ &
$3$      & $2$      & $0$      & $0$       \\\hline
\end{tabular}.
\end{center}
\caption{$X \subset S(1,2)$ where $b=2 \delta$ and $b=2\delta +1$,
respectively.}
\end{table}
Here, the Betti numbers lying in the vertical dots on Table 1 are as follows:\\

$\begin{array}{lll}
\quad\quad\quad\begin{tabular}{|c||c|c|c|c|c|}\hline
$\beta_{i,a+j}$ & $1$      & $3$      & $3$      & $1$      \\\hline
\end{tabular} &\text{ for $\delta+2 \leq j \leq b$ and}
\end{array}$
\smallskip

$\begin{array}{lll}
\quad\quad\quad\begin{tabular}{|c||c|c|c|c|c|}\hline
$\beta_{i,k}$         &
$0$      & $0$      & $0$      & $0$       \\\hline
\end{tabular}  &\quad\text{for $3 \leq k \leq a+\delta-1$}
\end{array}$
\end{theorem}

\begin{theorem}\label{thm:S(1,3)}
Let $X$ be an effective divisor of $S=S(1,3)$ in $\P^5$ linearly
equivalent to $aH+bF$ where either $a=0$ and $b \geq 4$ or else $a
\geq 1$ and $b \geq 2$. Then $\beta (X)$ is equal to the first
(resp. the second and the third) one of the following three tables
in Table 2 when $b=3 \delta -1$ (resp. $b=3 \delta$ and $b=3 \delta
+1$):

\begin{table}[hbt]\label{table:S(1,3)}
\begin{center}
\begin{tabular}{|c||c|c|c|c|c|c|}\hline
$\beta_{i,a+b}$ & $1$      & $4$      & $6$      & $4$   & $1$
\\\hline $\beta_{i,a+b-1}$          & $0$      & $0$      & $0$
& $0$  & $0$     \\\hline  $\vdots$
& $\vdots$ & $\vdots$ & $\vdots$ & $\vdots$ & $\vdots$ \\\hline
$\beta_{i,a+\delta+3}$ & $1$      & $4$      & $6$      & $4$   &
$1$   \\\hline $\beta_{i,a+\delta+2}$          & $0$      & $0$
& $0$      & $0$  & $0$     \\\hline $\beta_{i,a+\delta+1}$ & $1$
& $4$      & $10$      & $6$ & $1$     \\\hline $\beta_{i,a+\delta}$
& $2$      & $4$      & $0$      & $0$ & $0$     \\\hline $\vdots$
& $\vdots$ & $\vdots$ & $\vdots$ & $\vdots$ & $\vdots$  \\\hline
$\beta_{i,3}$          & $0$      & $0$      & $0$      & $0$  & $0$
\\\hline $\beta_{i,2}$          & $6$      & $8$      & $3$      &
$0$  & $0$     \\\hline
\end{tabular}$\quad$
\begin{tabular}{|c||c|c|c|c|c|c|}\hline
$\beta_{i,a+b}$ & $1$      & $4$      & $6$      & $4$   & $1$
\\\hline $\beta_{i,a+b-1}$          & $0$      & $0$      & $0$
& $0$  & $0$     \\\hline   $\vdots$
& $\vdots$ & $\vdots$ & $\vdots$ & $\vdots$ & $\vdots$ \\\hline
$\beta_{i,a+\delta+3}$ & $0$      & $0$      & $0$      & $0$  & $0$
\\\hline $\beta_{i,a+\delta+2}$ & $1$      & $4$      & $6$      &
$4$   & $1$   \\\hline $\beta_{i,a+\delta+1}$ & $0$      & $6$
& $8$      & $3$ & $0$     \\\hline $\beta_{i,a+\delta}$   & $1$
& $0$      & $0$      & $0$ & $0$     \\\hline $\vdots$
& $\vdots$ & $\vdots$ & $\vdots$ & $\vdots$ & $\vdots$  \\\hline
$\beta_{i,3}$          & $0$      & $0$      & $0$      & $0$  & $0$
\\\hline $\beta_{i,2}$          & $6$      & $8$      & $3$      &
$0$  & $0$     \\\hline
\end{tabular}
$\quad$
\begin{tabular}{|c||c|c|c|c|c|c|}\hline
$\beta_{i,a+b}$ & $1$      & $4$      & $6$      & $4$   & $1$
\\\hline $\beta_{i,a+b-1}$          & $0$      & $0$      & $0$
& $0$  & $0$     \\\hline $\vdots$
& $\vdots$ & $\vdots$ & $\vdots$ & $\vdots$ & $\vdots$ \\\hline
$\beta_{i,a+\delta+3}$ & $1$      & $4$      & $6$      & $4$   &
$1$   \\\hline $\beta_{i,a+\delta+2}$ & $0$      & $0$      & $0$
& $0$  & $0$     \\\hline $\beta_{i,a+\delta+1}$ & $4$      & $12$
& $12$      & $4$ & $0$     \\\hline $\beta_{i,a+\delta}$   & $0$
& $0$      & $0$      & $0$ & $0$     \\\hline $\vdots$
& $\vdots$ & $\vdots$ & $\vdots$ & $\vdots$ & $\vdots$  \\\hline
$\beta_{i,3}$          & $0$      & $0$      & $0$      & $0$  & $0$
\\\hline $\beta_{i,2}$          & $6$      & $8$      & $3$      &
$0$  & $0$     \\\hline
\end{tabular}
\end{center}
\caption{$X \subset S(1,3)$ where $b=3 \delta -1$, $b=3 \delta$ and
$b=3 \delta +1$, respectively.}
\end{table}
Here, the Betti numbers lying in the vertical dots table 2 are as follows:\\

$\begin{array}{lll}
\quad\quad\begin{tabular}{|c||c|c|c|c|c|c|}\hline
$\beta_{i,a+b-2j}$ & $1$      & $4$      & $6$      & $4$   & $1$
\\\hline
$\beta_{i,a+b-(2j+1)}$          & $0$      & $0$      & $0$
& $0$  & $0$     \\\hline
\end{tabular}& \quad\text{for $0 \leq j \leq \lceil\frac{b}{3}\rceil-2$ and }
\end{array}$
\smallskip

$\begin{array}{lll}
\quad\quad\begin{tabular}{|c||c|c|c|c|c|c|}\hline
$\beta_{i,k}$         &
$0$      & $0$      & $0$      & $0$ & $0$      \\\hline
\end{tabular}&\quad\quad\quad\quad \text{for $3 \leq k \leq a+\delta-1$}
\end{array}$
\end{theorem}
\smallskip

Next, let us consider the case of $a_1 =2$. Then $S$ and $X$ satisfy
(\ref{eq:1.2}) when $r \geq 6$ and $b \not\equiv 2 ~(\mod a_2)$. In
these cases, we can calculate $\beta (X)$ completely by Theorem
\ref{Thm:Decomposition 2} and Propositions \ref{prop:Betti number}.
In Section 5, we solve Problem ($\dagger$) for the missing case
where $r \geq 6$ and $b \equiv 2 ~(\mod a_2)$. In consequence, we
solve Problem ($\dagger$) completely when $a_1 =2$ and $a_2 \geq 3$.
See Theorem \ref{thm:a2is2} for details. In Example \ref{ex:S(2,3)},
we apply this result to the case of $S=S(2,3)$, and as in Theorem
\ref{thm:S(1,2)} and Theorem \ref{thm:S(1,3)}, we obtain an explicit
description of $\beta (X)$ for every $X$ considered in Problem
($\dagger$). It turns out that $\beta (X)$ has six different types.
\\

\noindent {\bf Organization of the paper.} In Section 2, we recall
some definitions and basic facts. In Section 3, we prove Theorem
\ref{Thm:Decomposition 1} and Theorem \ref{Thm:Decomposition 2},
which are our main results in this paper. Section 4 is devoted to
give a complete description of $\beta \left(E(r,s,t) \right)$. In
Section 5, we apply our results in the previous sections to the
cases where $a_1 =1$, $a_1 =2$ and $a_1 = a_2$. Also we present some
examples that illustrate how our results can be applied to specific
cases. At the end of Section 5, we provide some examples which
show that the statements in Theorem \ref{Thm:Decomposition 1} and Theorem \ref{Thm:Decomposition 2} are sharp.   \\

\noindent {\bf Acknowledgement.} The first named author was
supported by Basic Science Research Program through the National
Research Foundation of Korea(NRF) funded by the Ministry of
Education(NRF-2017R1D1A1B03031438). The second named author was
supported by the Korea Research Foundation Grant funded by the
Korean Government (NRF-2018R1D1A1B07041336).\\

\section{Preliminaries}
\noindent In this section, we recall some definitions and basic
facts.

\begin{notation and remark}
Let $R := \Bbbk [x_0 , \ldots , x_r ]$ be the homogeneous coordinate
ring of the projective $r$-space $\P^r$ defined over an
algebraically closed field $\Bbbk$ of arbitrary characteristic.
\smallskip

\noindent (1) For a non-zero finitely generated graded $R$-module
$M$, the graded Betti numbers are defined by $\beta_{i,j}(M) := {\rm
dim}_{\Bbbk} ~ {\rm Tor}^R_i (M,\Bbbk)_{i+j}$. The Betti table of
$M$, denoted by $\beta (M)$, is the table whose entry in the $i$-th
column and $j$-th row is $\beta_{i,j}(M)$. Throughout this paper, we
present
$\beta (M)$ as follows:\\

\begin{center}
$\beta (M)=$
\begin{tabular}{|c||c|c|c|c|c|c|c|}\hline
$i$               & $0$               & $1$               & $\cdots$         & $i$      & $\cdots$ & $r-1$ & $r$ \\\hline
$\vdots$          & $\vdots$          & $\vdots$          & $\ddots$         & $\vdots$ & $\ddots$ & $\vdots$ & $\vdots$ \\\hline
$\beta_{i,j} (M)$ & $\beta_{0,j} (M)$ & $\beta_{1,j} (M)$ & $\cdots$         & $\beta_{i,j} (M)$ & $\cdots$ & $\beta_{r-1,j} (M)$ & $\beta_{r,j}(M)$  \\\hline
$\vdots$          & $\vdots$          & $\vdots$          & $\ddots$         & $\vdots$  & $\ddots$ & $\vdots$        & $\vdots$              \\\hline
$\beta_{i,1} (M)$ & $\beta_{0,1} (M)$ & $\beta_{1,1} (M)$ & $\cdots$          & $\beta_{i,1} (M)$ & $\cdots$ & $\beta_{r-1,1} (M)$ & $\beta_{r,1} (M)$  \\\hline
$\beta_{i,0} (M)$ & $\beta_{0,0} (M)$ & $\beta_{1,0} (M)$ & $\cdots$          & $\beta_{i,0} (M)$ & $\cdots$ & $\beta_{r-1,0} (M)$ & $\beta_{r,0} (M)$  \\\hline
$\vdots$          & $\vdots$          & $\vdots$          & $\ddots$ & $\vdots$     & $\ddots$ & $\vdots$ & $\vdots$ \\\hline
\end{tabular}
\end{center}
\smallskip

\noindent (2) Let $\mathbb{B}_r := \bigoplus_{-\infty} ^{\infty}
\mathbb{Z} ^{r+1}$ be the additive group of all tables with $r+1$
columns whose entries are integers. We regard $\beta (M)$ as an
element of $\mathbb{B}_r$.
\smallskip

\noindent (3) Let $\ell \in \Z$ and $T \in \mathbb{B}_r$. Then we
denote by $T[\ell]$ the table obtained by lifting $T$ up to $\ell$
rows. That is, the $(i,j)$-th entry of $T[\ell]$ is exactly equal to
the $(i,j-\ell)$-th entry of $T$.
\smallskip

\noindent (4) For a closed subscheme $X \subset \P^r$, we will
denote by $\beta (X)$ the Betti table of the homogeneous ideal
$I(X)$ of $X$ as a graded $R$-module.
\end{notation and remark}

Let $r$ and $s$ be integers such that $r \geq 3$ and $1 \leq s \leq
r$. Consider a (possibly degenerate) rational normal curve $S(s)
\subset \P^r$ of degree $s$. For each integer $t \geq 2$, we denote
by $\mathcal{F}_t$ the line bundle on $S(s)$ of degree $-t$. Also we
define $E(r,s,t)$ as the graded $R$-module associated to
$\mathcal{F}_t$. That is,
\begin{equation*}
E(r,s,t) = \bigoplus_{m \in \Z} H^0 (\P^r ,\mathcal{F}_t \otimes
\mathcal{O}_{\P^r} (m)).
\end{equation*}
In Section 4, we get a complete description of $\beta \left(
E(r,s,t) \right)$ for all $r$, $s$ and $t$. For details, see
Proposition \ref{prop:Betti number}.

We finish this section by investigating a few basic properties of
$E(r,s,t)$. If $x$ is a real number, let $\lceil x \rceil$ denote
the smallest integer $\geq x$.

\begin{lemma}\label{lem:basics of E(r,s,t)}
Let $r,s,t$ and $E(r,s,t)$ be as above. Then,

\renewcommand{\descriptionlabel}[1]%
             {\hspace{\labelsep}\textrm{#1}}
\begin{description}
\setlength{\labelwidth}{13mm} \setlength{\labelsep}{1.5mm}
\setlength{\itemindent}{0mm}

\item[{\rm (1)}] For any integer $n \in \Z$, it holds that
\begin{equation*}
E(r,s,t+ns) \cong E(r,s,t)(-n) \quad \mbox{and hence}\quad \beta
\left( E(r,s,t+ns) \right) = \beta \left( E(r,s,t) \right) [n].
\end{equation*}

\item[{\rm (2)}] ${\rm reg}(E(r,s ,t)) = \left\lceil \frac{t-1}{s}
\right\rceil +1$.
\smallskip

\item[{\rm (3)}] $\beta_{i,j} (E(r,s,t)) = 0$ if $j \neq \left\lceil \frac{t-1}{s}
\right\rceil, \left\lceil \frac{t-1}{s} \right\rceil +1$. Furthermore, if $t \equiv 1$ $($mod $s)$, then $\beta_{i,j} (E(r,s,t)) = 0$ if $j \neq \left\lceil \frac{t-1}{s} \right\rceil +1$.
\end{description}
\end{lemma}

\begin{proof}
{\rm (1)} The assertions come from the fact that $\mathcal{F}_{t+ns}
\cong \mathcal{F}_t \otimes \mathcal{O}_{\P^r} (-n)$.
\smallskip

\noindent {\rm (2)} Write $t = s \times u + p$ for some $2 \leq p
\leq s +1$. Thus $u = \left\lceil \frac{t-1}{s} \right\rceil -1$.
Then, by (1), we have
\begin{equation*}
{\rm reg}(E(r,s , t)) = {\rm reg}(E(r,s , p)) + \left\lceil
\frac{t-1}{s} \right\rceil -1.
\end{equation*}
Also the line bundle $\mathcal{O}_{\P^1} (-p)$ on $S (s)$ is
$2$-regular as a coherent sheaf on $\P^r$. This completes the proof.
\smallskip

\noindent {\rm (3)} It is obvious that $\beta_{i,j} (E(r,s,t)) = 0$
if $j > {\rm reg}(E(r,s ,t))$. For the remaining cases, we recall
that for every $n \in \Z$,
\begin{equation*}
E(r,s ,t)_n = H^0 (\P^r ,\mathcal{F}_t \otimes \mathcal{O}_{\P^r}
(n)) \cong H^0 (\P^1 , \mathcal{O}_{\P^1} (ns-t)).
\end{equation*}
Therefore $E(r,s ,t)_n = 0$ if and only if $n \leq \frac{t-1}{s}$.
In particular, we get $E(r,s ,t)_n = 0$ if $n \leq {\rm reg}(E(r,s
,t))-2$. This implies that $\beta_{i,j} (E(r,s,t)) = 0$ if $j \leq
{\rm reg}(E(r,s ,t))-2$. For the last statement, suppose that $t \equiv 1$ (mod $s$). Then it holds that
\begin{equation*}
\frac{t-1}{s} = {\rm reg}(E(r,s ,t))-1.
\end{equation*}
Therefore $E(r,s ,t)_n = 0$ if $n \leq {\rm reg}(E(r,s ,t))-1$. In
particular, $\beta_{i,j} (E(r,s,t)) = 0$ if $j = {\rm reg}(E(r,s
,t))-1$.
\end{proof}

\section{Decomposition Theorems of $\beta (X)$}
\noindent Throughout this section, let $S=S(a_1 , a_2 ) \subset
\P^r$ be a smooth rational normal surface scroll and $X$ be an
effective divisor of $S$ linearly equivalent to $aH+bF$ for some
$a,b \in \Z$ such that either $a=0$ and $b > a_2$ or else $a \geq 1$
and $b \geq 2$. We denote by $C_0$ the minimal section $S(a_1 )$ of
$S$. Note that $C_0$ is linearly equivalent to $H-a_2 F$.

The aim of this section is to prove two theorems about the
decomposition of $\beta(X)$ into the sum of several Betti diagrams
that are much simpler.

To state our results about $\beta (X)$, we need the integers $\delta = \delta (X)$,
$\epsilon = \epsilon (X)$ and $q_{\ell} = q_{\ell} (X)$ for $1 \leq
\ell \leq \delta$ which are defined in terms of $a_1$, $a_2$, $a$
and $b$ as
\begin{equation*}
\delta := \left\lceil \frac{b-1}{a_2 } \right\rceil, \quad \epsilon
:= b-(\delta-1)a_2 \quad \mbox{and} \quad q_{\ell} = a_1 a+ b + (a_1
- a_2 ) (\ell -1)  \quad \mbox{for} \quad 1 \leq \ell \leq \delta.
\end{equation*}
Therefore the integer $\delta (X)$ can be regarded as a measure of how far $X$ is from the arithmetically Cohen-Macaulay property since a curve linearly equivalent to $X + \ell C_0$ is ACM if and only if $\ell = \delta$. Indeed, see \cite[Theorem 4.3]{P2} for the cases where $a\geq 1$ or $a=0$ and $b > a_2 +1$. Also, if $a=0$ and $b=a_2+1$ then $\delta (X)=1$ and $X + C_0$ is a rational normal curve which is apparently arithmetically Cohen-Macaulay. Also $q_{\ell} (X)$ is the intersection number of $X + (\ell-1) C_0$ and $C_0$. \\

\begin{theorem}\label{Thm:Decomposition 1}
Suppose that $a_2 \geq 2 a_1 -1$. Then
\begin{equation}\label{eq:Decomposition 1}
\beta (X)= \beta (S) +   \beta \left( E(H+ \epsilon F ) \right)
[a+\delta-2] + \sum_{\ell=1} ^{\delta -1} \beta \left( E(r,a_1,
q_{\ell} ) \right ) .
\end{equation}
\end{theorem}

\begin{theorem}\label{Thm:Decomposition 2}
Suppose that $a_2 \geq 2 a_1 -1$ and $a_1 +1 \leq \epsilon \leq a_2
+1$. Then
\begin{equation}\label{eq:Decomposition 2}
\beta (X)= \beta (S) + \beta \left( E(r-1,r-1 ,a_1 + \epsilon )
\right )[a+\delta-1] + \sum_{\ell=1} ^{\delta} \beta \left(
E(r,a_1,q_{\ell} ) \right ).
\end{equation}
\end{theorem}

The following proposition plays a cornerstone in proving the above
two theorems.

\begin{proposition}\label{prop:fundamental exact sequence}
Let $X \equiv aH+bF$ be an effective divisor of $S$ satisfying
$($\ref{eq:1.1}$)$ and let $Y$ be the scheme-theoretic union of $X$
and $C_0$. Also put $q := aa_1 +b$. Then

\renewcommand{\descriptionlabel}[1]%
             {\hspace{\labelsep}\textrm{#1}}
\begin{description}
\setlength{\labelwidth}{13mm} \setlength{\labelsep}{1.5mm}
\setlength{\itemindent}{0mm}

\item[{\rm (1)}] ${\rm reg}(X) = {\rm reg}(E(r,a_1 , q))$.

\item[{\rm (2)}]  ${\rm reg}(X) \geq {\rm reg}(Y)$. Moreover, ${\rm reg}(X) > {\rm reg}(Y)$ in the following
cases:
\begin{enumerate}
\item[$(i)$] $a_1 +2 \leq b \leq a_2 +1$;
\item[$(ii)$] $a_2 \geq a_1 +1$, $b \geq a_2 +2$ and $b \equiv \gamma ~(\mbox{mod}~a_1 )$ for some
$2 \leq \gamma \leq a_2 -a_1 +1$.
\end{enumerate}
\smallskip

\item[{\rm (3)}] Suppose that $X$ does not contain $C_0$ as a component. Then there is an exact sequence of graded $R$-modules
\begin{equation*}
0 \rightarrow I(Y) \rightarrow I(X) \rightarrow E(r,a_1 ,q )
\rightarrow  0
\end{equation*}
where $I(X)$ and $I(Y)$ are respectively the homogeneous ideals of
$X$ and $Y$ in $R$.
\smallskip

\item[{\rm (4)}] If ${\rm reg}(X) > {\rm reg}(Y)$ or ${\rm reg}(X) = {\rm reg}(Y)$ and $b \equiv 1 ~(\mbox{mod}~a_1 )$, then
\begin{equation}\label{eq:2.1}
\beta (X)=\beta (Y) + \beta (E(r, a_1 , q )).
\end{equation}

\item[{\rm (5)}] The above $($\ref{eq:2.1}$)$ holds if one of the following conditions holds:
\begin{enumerate}
\item[$(i)$] $a_1 +1 \leq b \leq a_2 +1$;
\item[$(ii)$] $a_2 \geq a_1$, $b \geq a_2 +2$ and $b \equiv \gamma ~(\mbox{mod}~a_1 )$ for some $1 \leq \gamma \leq a_2 -a_1
+1$.
\end{enumerate}
\end{description}
\end{proposition}

\begin{proof}
(1) By Lemma \ref{lem:basics of E(r,s,t)}.(2), we get ${\rm
reg}(E(r,a_1 , q)) = \left\lceil \frac{q-1}{a_1} \right\rceil +1$.
Also ${\rm reg}(X) = a+1+ \left\lceil \frac{b-1}{a_1} \right\rceil$
by \cite[Theorem 4.3]{P2}. Therefore ${\rm reg}(X) = {\rm
reg}(E(r,a_1 , q))$.
\smallskip

\noindent (2) Since $C_0 \equiv H-a_2 F$, the divisor class of $Y$
is equal to $(a+1)H+(b-a_2 )F$. Thus we get
\begin{equation*}
{\rm reg}(Y) = \begin{cases} a+2 & \mbox{if $2 \leq b \leq a_2 +1$,
and}\\
a+2 + \left\lceil \frac{b-a_2 -1}{a_1} \right\rceil & \mbox{if $b
\geq a_2 +2$.} \end{cases}
\end{equation*}
by \cite[Theorem 4.3]{P2}.

When $2 \leq b \leq a_2 +1$, we have
\begin{equation*}
{\rm reg}(X)-{\rm reg}(Y) = \left\lceil \frac{b-1}{a_1} \right\rceil
- 1 \geq 0.
\end{equation*}
In particular, ${\rm reg}(X)>{\rm reg}(Y)$ if and only if $a_1 + 2
\leq b \leq a_2 +1$.

Now, suppose that $b \geq a_2 +2$. Then
\begin{equation}\label{eq:difference of regularity}
{\rm reg}(X)-{\rm reg}(Y) = \left\lceil \frac{b-1}{a_1} \right\rceil
- \left\lceil \frac{b-a_2 -1}{a_1} \right\rceil -1
\end{equation}
Thus ${\rm reg}(X)={\rm reg}(Y)$ if $a_1 = a_2$. Also if $a_2 >
a_1$, then it holds that
\begin{equation}\label{eq:estimation}
\left\lceil \frac{b-1}{a_1} \right\rceil \geq \left\lceil
\frac{b-a_2 -1}{a_1} \right\rceil + \left\lceil \frac{a_2}{a_1}
\right\rceil -1 \geq \left\lceil \frac{b-a_2 -1}{a_1} \right\rceil
+1
\end{equation}
and hence ${\rm reg}(X) \geq {\rm reg}(Y)$. Moreover, if $a_2 \geq
2a_1$ or if $a_1 +1 \leq a_2 \leq 2a_1 -1$ and $b \equiv \gamma
~(\mbox{mod}~a_1 )$ for some $2 \leq \gamma \leq a_2 -a_1 +1$, then
one can check that
\begin{equation}\label{eq:estimation}
\left\lceil \frac{b-1}{a_1} \right\rceil - \left\lceil \frac{b-a_2
-1}{a_1} \right\rceil  \geq 2
\end{equation}
and hence ${\rm reg}(X) > {\rm reg}(Y)$.
\smallskip

\noindent (3) The ideal sheaf $\mathcal{I}_Y$ of $Y$ in $\P^r$ is
equal to $\mathcal{I}_X \cap \mathcal{I}_{C_0}$ and hence the quotient
$\mathcal{I}_X /\mathcal{I}_Y$ is isomorphic to $\mathcal{I}_{\Gamma
/ C_0}$ where $\Gamma$ is the scheme-theoretic intersection of $X$
and $C_0$. Thus we have the exact sequence
\begin{equation}\label{exact sequence of ideal sheaves}
0 \rightarrow \mathcal{I}_Y \rightarrow \mathcal{I}_X \rightarrow
\mathcal{O}_{C_0} (-\Gamma) \rightarrow 0
\end{equation}
of coherent sheaves on $\P^r$. Note that the length of $\Gamma$ is
equal to $q$. Thus the graded $R$-module associated to
$\mathcal{O}_{C_0} (-\Gamma)$ is equal to $E(r,a_1 , q)$. From
(\ref{exact sequence of ideal sheaves}), we get the cohomology long
exact sequence
\begin{equation*}
0 \rightarrow H^0(\P^r, \mathcal{I}_Y (j)) \rightarrow H^0(\P^r,
\mathcal{I}_X (j)) \overset{\varphi_j}{\rightarrow} E(r,a_1 ,q )_j
\rightarrow H^1(\P^r, \mathcal{I}_Y (j)) \rightarrow \cdots
\end{equation*}
for every $j \in \Z$. Thus it needs to check that $\varphi_j$ is
always surjective. In the proof of Lemma \ref{lem:basics of
E(r,s,t)}.(3), it is shown that $E(r,a_1 , q)_n =0$ for $n \leq {\rm reg}(E(r,a_1 , q))-2= {\rm
reg}(X)-2$ by (1). Also, by (1) and (2), we get
\begin{equation*}
H^1(\P^r, \mathcal{I}_Y (j))=0 \quad \mbox{for} \quad j \geq {\rm
reg}(X)-1.
\end{equation*}
In consequence, it is shown that $\varphi_j$ is surjective for all
$j \in \Z$.
\smallskip

\noindent (4) From (3), we get the long exact sequence
\begin{equation}\label{exs:tor2}
\dots \rightarrow {\rm Tor}^R_{i+1}(E(r,a_1,q),\Bbbk)_{i+j}
\rightarrow {\rm Tor}^R_i(I(Y),\Bbbk)_{i+j} \rightarrow {\rm
Tor}^R_i(I(X),\Bbbk)_{i+j}
\end{equation}
\begin{equation*}
\rightarrow {\rm Tor}^R_{i}(E(r,a_1,q),\Bbbk)_{i+j} \rightarrow {\rm
Tor}^R_{i-1}(I(Y),\Bbbk)_{i+j} \rightarrow \cdots .
\end{equation*}
Thus it suffices to show that
\begin{equation}\label{eq:canonical isom}
{\rm Tor}^R_i(I(X),\Bbbk)_{i+j} \cong {\rm
Tor}^R_i(I(Y),\Bbbk)_{i+j} \bigoplus {\rm
Tor}^R_{i}(E(r,a_1,q),\Bbbk)_{i+j}
\end{equation}
as $\Bbbk$-vector spaces for all $i \geq 0$ and $j \leq {\rm
reg}(X)$.

Firstly, suppose that ${\rm reg}(X) > {\rm reg}(Y)$. When $j= {\rm
reg}(X)$ we get
\begin{equation*}
{\rm Tor}^R_i(I(Y),\Bbbk)_{i+j} = {\rm
Tor}^R_{i-1}(I(Y),\Bbbk)_{i+j}=0 .
\end{equation*}
When $j = {\rm reg}(X)-1$ it follows by Lemma \ref{lem:basics of
E(r,s,t)}.(3) that
\begin{equation*}
{\rm Tor}^R_{i+1}(E(r,a_1,q),\Bbbk)_{i+j}={\rm
Tor}^R_{i-1}(I(Y),\Bbbk)_{i+j} =0.
\end{equation*}
Also, when $j \leq {\rm reg}(X)-2$ we get
\begin{equation*}
{\rm Tor}^R_{i+1}(E(r,a_1,q),\Bbbk)_{i+j}={\rm
Tor}^R_{i}(E(r,a_1,q),\Bbbk)_{i+j}=0.
\end{equation*}
In consequence, (\ref{eq:canonical isom}) is verified if ${\rm
reg}(X) > {\rm reg}(Y)$.

Now, suppose that ${\rm reg}(X) = {\rm reg}(Y)$ and $b \equiv 1$
(mod $a_1$). Then
\begin{equation*}
{\rm Tor}^R_{i+1}(E(r,a_1,q),\Bbbk)_{i+j}=0 \quad \mbox{for all $j
\neq {\rm reg}(E(r,a_1,q))$}
\end{equation*}
by Lemma \ref{lem:basics of E(r,s,t)}.(3). Therefore it holds that
\begin{equation*}
{\rm Tor}^R_{i+1}(E(r,a_1,q),\Bbbk)_{i+j} = {\rm
Tor}^R_{i-1}(I(Y),\Bbbk)_{i+j}=0 \quad \mbox{if $j= {\rm reg}(X)$}
\end{equation*}
and
\begin{equation*}
{\rm Tor}^R_{i+1}(E(r,a_1,q),\Bbbk)_{i+j} = {\rm Tor}^R_i
(E(r,a_1,q),\Bbbk)_{i+j}=0 \quad \mbox{if $j < {\rm reg}(X)$}.
\end{equation*}
This completes the proof of (\ref{eq:canonical isom}) when ${\rm
reg}(X) = {\rm reg}(Y)$ and $b \equiv 1$ (mod $a_1$).
\smallskip

\noindent (5) By (2) and (4), (\ref{eq:2.1}) holds if $(i)$ or
$(ii)$ in (2) holds. Thus it remains to consider the cases where
either
\begin{enumerate}
\item[$(i)$] $b=a_1 +1$ or
\item[$(ii)$] $a_2 \geq a_1$, $b \geq a_2 +2$ and $b \equiv 1$ (mod
$a_1$).
\end{enumerate}
In these cases, we have either ${\rm reg}(X)>{\rm reg}(Y)$ or else
${\rm reg}(X)={\rm reg}(Y)$ and $b \equiv 1$ (mod $a_1$). Therefore
(\ref{eq:2.1}) holds by (4).
\end{proof}
\smallskip

\noindent {\bf Proof of Theorem \ref{Thm:Decomposition 1}.} If $a \geq
1$, then the line bundle $\mathcal{O}_S (X)$ on $S$ is very ample
and hence there exists a smooth irreducible curve, say $X'$. Since
$\beta(X)=\beta(X' )$ by \cite[Proposition 4.1(1)]{P2}, we may
assume that $X$ does not contain $C_0$ as a component.
\smallskip

\noindent (a) We will prove our theorem by induction on $\delta$.
When $\delta=1$, we need to show that
\begin{equation*}
\beta (X)= \beta (S) + \beta \left(  E(H+ \epsilon F ) \right)
[a-1].
\end{equation*}
To this aim, let $Z = H+ \epsilon F$ and consider the two short
exact sequences
\begin{equation*}
0 \rightarrow \mathcal{I}_S \rightarrow \mathcal{I}_X \rightarrow
\mathcal{O}_S (-X) \rightarrow 0
\end{equation*}
and
\begin{equation*}
0 \rightarrow \mathcal{I}_S \rightarrow \mathcal{I}_Z \rightarrow
\mathcal{O}_S (-Z) \rightarrow 0.
\end{equation*}
Then we have the following two short exact sequences of $R$-modules
\begin{equation*}
0 \rightarrow I(S) \rightarrow I(X) \rightarrow E (X) \rightarrow 0
\quad \mbox{and} \quad 0 \rightarrow I(S) \rightarrow I(Z)
\rightarrow E (Z) \rightarrow 0
\end{equation*}
where $E (X)$ (resp. $E (Z)$) denotes the graded $R$-module
associated to $\mathcal{O}_S (-X)$ (resp. $\mathcal{O}_S (-Z)$).
Since $X-Z \equiv (a-1)H$, it holds that
\begin{equation*}
E (X) \cong E(Z) (-a+1) \quad \mbox{and hence} \quad  \beta (E (X) )
= \beta (E (Z) ) [a-1].
\end{equation*}
Also, by \cite[Proposition 3.2]{P2}, it holds that $\beta (X) = \beta
(S)+ \beta (E (X) )$. Consequently, we get
\begin{equation*}
\beta (X) = \beta (S)+ \beta (E (X) ) = \beta (S)+ \beta (E (Z)
)[a-1] =\beta (S) + \beta \left( E(H+ \epsilon F )  \right) [a-1].
\end{equation*}
Now, suppose that $\delta >1$ and let $Y=X \cup C_0$. Note that
since $a_2 \geq 2a_1 -1$ we can apply Proposition
\ref{prop:fundamental exact sequence}.(5) to $X$ whenever $b \geq
a_2 +2$. Therefore we have
\begin{equation}\label{eq:2.8}
\beta (X) = \beta (Y)+ \beta (E (r,a_1 ,q_1 ) ).
\end{equation}
Observe that $\delta (Y) = \delta (X) -1$, $\epsilon (X)=\epsilon
(Y)$ and $q_{\ell} (Y) = q_{\ell+1} (X)$ for $1 \leq \ell \leq
\delta (Y)$. By induction hypothesis, we have
\begin{equation}\label{eq:2.9}
\beta (Y)= \beta (S) + \beta \left(  E(H+ \epsilon F )  \right)
[(a+1)+\delta(Y)-2] + \sum_{\ell=1} ^{\delta(Y) -1} \beta \left(
E(r,a_1,q_{\ell} (Y) ) \right ).
\end{equation}
Now, the desired formula (\ref{eq:Decomposition 1}) comes by combining
(\ref{eq:2.8}) and (\ref{eq:2.9}).  \qed \\
\smallskip

\noindent {\bf Proof of Theorem \ref{Thm:Decomposition 2}.} To obtain the formula (\ref{eq:1.2}), we focus on the
term $\beta \left(  E(H+ \epsilon F )  \right)$ in (\ref{eq:1.1}).
Let $M$ be an irreducible curve on $S$ linearly equivalent to
$H+\epsilon F$. Thus
\begin{equation*}
\beta (M) = \beta (S)+ \beta \left(  E(H+ \epsilon F )  \right)
\end{equation*}
by \cite[Proposition 3.2]{P2}. Now, let $N$ be the scheme-theoretic
union of $M$ and $C_0$. Then we can apply Proposition
\ref{prop:fundamental exact sequence}.(5.i) to our case since $a_1
+1 \leq \epsilon \leq a_2 +1$. That is,
\begin{equation}\label{eq:2.10}
\beta (M)=  \beta (N) + \beta \left( E(r,a_1,  a_1  + \epsilon  )
\right ).
\end{equation}
Therefore we get
\begin{equation*}
\beta \left(  E(H+ \epsilon F )  \right) = - \beta (S)+ \beta (N) +
\beta \left( E(r,a_1,  a_1  + \epsilon ) \right ).
\end{equation*}
Now, observe that $N \equiv 2H + (\epsilon - a_2 )F$ and hence $N
\subset \P^r$ is arithmetically Cohen-Macaulay (cf. \cite[Theorem
4.3]{P2}. Let $\Gamma \subset \P^{r-1}$ be a general hyperplane
section of $N$. Then $\Gamma$ is contained in $S(r-1)$ since $N$ is
a divisor of the rational normal surface scroll $S$. Also $|\Gamma|
= {\rm deg}(N)=2a_1 + a_2 + \epsilon$. Therefore we get
\begin{equation*}
\beta (\Gamma) = \beta (S(r-1)) + \beta (E(r-1,r-1,2a_1 + a_2 +
\epsilon))
\end{equation*}
by \cite[Proposition 3.2]{P2}. Since $\beta (\Gamma) = \beta (N)$,
$\beta (S(r-1)) = \beta (S)$ and $r-1 = a_1 + a_2$, it follows that
\begin{equation*}
\beta (N) - \beta (S) = \beta (E(r-1,r-1,2a_1 + a_2 + \epsilon)) =
\beta (E(r-1,r-1, a_1  + \epsilon))[1].
\end{equation*}
In consequence, it is shown that
\begin{equation}\label{eq:2.11}
\beta \left(  E(H+ \epsilon F ) \right) = \beta (E(r-1,r-1, a_1 +
\epsilon ))[1] + \beta \left( E(r,a_1, a_1  + \epsilon ) \right ).
\end{equation}
By using the definitions of $q_{\delta}$ and $\epsilon$, one can
check that
\begin{equation}\label{eq:2.12}
\beta \left( E(r,a_1, \epsilon+a_1 ) \right ) [a+\delta-2] = \beta
\left( E(r,a_1, q_{\delta} ) \right ).
\end{equation}
Now, we get the desired formula (\ref{eq:Decomposition 2}) by combining (\ref{eq:Decomposition 1}), (\ref{eq:2.11}) and (\ref{eq:2.12}).   \qed \\

\section{The graded Betti numbers of the module $E(r,s,t)$}
\noindent In this section, we calculate the Betti numbers of the
graded R-module $E(r,s,t)$, which is defined to be the graded
$R$-module associated to the line bundle $\mathcal{O}_{\P^1} (-t)$
on a rational normal curve $S(s)$ of degree $s$ in $\P^r$.

\begin{proposition}\label{prop:Betti number}
Suppose that $t=p + \ell s$ for some $2 \leq p \leq s +1$. Then
$\beta_{i,j} (E(r,s ,t)) = 0$ if $j \neq \left\lceil \frac{t-1}{s}
\right\rceil  , \left\lceil \frac{t-1}{s} \right\rceil+1$. Also
\begin{equation*}
\beta_{i,\ell+1} (E(r,s ,t)) = \overset{s
+1-p}{\underset{k=0}{\sum}}(s +1-p-k){{s }\choose{k}}{{r-s
}\choose{i-k}}
\end{equation*}
and
\begin{equation*}
\beta_{i,\ell+2} (E(r,s ,t)) = \overset{i+1}{\underset{k=s
+2-p}{\sum}}(k+p-s -1){{s}\choose{k}}{{r-s }\choose{i+1-k}}.
\end{equation*}
\end{proposition}
\smallskip

We will give a proof of Proposition \ref{prop:Betti number} at the
end of this section.

In the following corollaries, we show how all the Betti tables in
the right hand side of (\ref{eq:Decomposition 2}) in Theorem
\ref{Thm:Decomposition 2} can be obtained from Proposition
\ref{prop:Betti number}.

\begin{corollary}\label{cor:SMD}
Let $S \subset \P^r$ be a surface of minimal degree. Then
$\mbox{reg}(S)=2$ and $\beta (S)$ is of the form
\begin{equation*}
\beta (S) =
\begin{tabular}{|c||c|c|c|c|c|}\hline
$\beta_{i,2} (S)$ & $\beta_{0,2} (S)$ & $\beta_{1,2}(S)$ & $\cdots$
& $\beta_{r-1,2} (S)$ & $\beta_{r,2} (S)$  \\\hline
\end{tabular}
\end{equation*}
where $\beta_{i,2} (S) = (i+1) {{r-1} \choose {i+2}}$ for all $i
\geq 0$.
\end{corollary}

\begin{proof}
Let $C \subset \P^{r-1}$ be a general hyperplane section of $S$.
Since $S$ is arithmetically Cohen-Macaulay, it holds that $\beta (S)
= \beta (C)$. Also the homogeneous coordinate ring $A_C$ of $C$ is
equal to $E(r-1 , r-1 , r-1)(-1)$ since $C$ is projectively normal.
Therefore we have
\begin{equation*}
\beta_{i,2} (S) = \beta_{i,2} (C) = \beta_{i+1,1} \left( A_C \right)
= \beta_{i+1,2} \left( E(r-1 , r-1 , r-1) \right) = (i+1) {{r-1}
\choose {i+2}}
\end{equation*}
for all $i \geq 0$ by Proposition \ref{prop:Betti number}.
\end{proof}

\begin{corollary}\label{cor:ACM divisor}
Suppose that $2 \leq p \leq  r$. Then $\beta \left( E(r-1,r-1,p)
\right)$ is of the form
\begin{equation*}
\beta \left( E(r-1,r-1,p) \right) =
\begin{tabular}{|c||c|c|c|c|c|}\hline
$\beta_{i,2}$ & $\beta_{0,2} $ & $\beta_{1,2} $ & $\cdots$ &
$\beta_{r-1,2}$ & $\beta_{r,2}$  \\\hline $\beta_{i,1}$ &
$\beta_{0,1}$ & $\beta_{1,1}$ & $\cdots$ & $\beta_{r-1,1} $ &
$\beta_{r,1}$
\\\hline
\end{tabular}
\end{equation*}
where
\begin{equation*}
\beta_{i,1} = \begin{cases} (r-p-i){{r-1 }\choose{i}} \quad & \mbox{for $0 \leq i \leq r-1-p$, }\\
                            0                                       \quad & \mbox{for $r-p  \leq i$ } \end{cases}
\end{equation*}
and
\begin{equation*}
\beta_{i,2} = \begin{cases} 0                             \quad & \mbox{for $0 \leq i \leq r-1-p$, }\\
                 (i+1+p-r){{r-1 }\choose{i+1}}   \quad & \mbox{for $r-p  \leq i$.} \end{cases}
\end{equation*}
\end{corollary}

\begin{proof}
The desired formulas are directly proved by Proposition
\ref{prop:Betti number}.
\end{proof}

\begin{corollary}\label{cor:t=s+1}
$\beta \left( E(r,s,s+1) \right)$ is of the form
\begin{equation*}
\beta \left( E(r,s,s+1) \right) =
\begin{tabular}{|c||c|c|c|c|c|}\hline
$\beta_{i,2} $ & $\beta_{0,2}$ & $\beta_{1,2}$ & $\cdots$ &
$\beta_{r-1,2}$ & $\beta_{r,2}$  \\\hline
\end{tabular}
\end{equation*}
where $\beta_{i,2} = s {{r-1} \choose {i}}$ for all $i \geq 0$.
\end{corollary}

\begin{proof}
This comes immediately from Proposition \ref{prop:Betti number}.
\end{proof}
\smallskip

\begin{example}\label{ex:E(r,2,p)}
By Proposition \ref{prop:Betti number} and Corollary
\ref{cor:t=s+1}, we obtain $\beta(E(r,2,2))$ and $\beta(E(r,2,3))$.
More precisely, we have
\begin{center}
$\beta (E(r,2,2)) =$
\begin{tabular}{|c||c|c|c|c|c|c|c|}\hline
$\beta_{i,2}$ &$\beta_{0,2}$ & $\beta_{1,2}$ & $\cdots$ &
$\beta_{i,2}$ & $\cdots$ & $\beta_{r-1,2}$ & $\beta_{r,2}$  \\\hline
$\beta_{i,1}$ &$\beta_{0,1}$ & $\beta_{1,1}$ & $\cdots$ &
$\beta_{i,1}$ & $\cdots$ & $\beta_{r-1,1}$ & $\beta_{r,1}$  \\\hline
\end{tabular}
\end{center}
where
\begin{equation*}
\beta_{i,1} = \beta_{i,1} (E(r,2,2))= {{r-2} \choose {i}} \quad
\mbox{and} \quad \beta_{i,2} =\beta_{i,2} (E(r,2,2))= {{r-2} \choose
{i-1}}
\end{equation*}
and
\begin{center}
$\beta (E(r,2,3)) =$
\begin{tabular}{|c||c|c|c|c|c|c|c|}\hline
$\beta_{i,2}$ & $\beta_{0,2}$ & $\beta_{1,2}$ & $\cdots$ &
$\beta_{i,2}$ & $\cdots$ & $\beta_{r-1,2}$ & $\beta_{r,2}$  \\\hline
\end{tabular}
\end{center}
where
\begin{equation*}
\beta_{i,2} =  \beta_{i,2} (E(r,2,3))= 2{{r-1} \choose {i}}.
\end{equation*}
\end{example}

\begin{remark}\label{rmk:duality}
By Proposition \ref{prop:Betti number}, we need to compute $a_1$
distinct tables for given $r$ and $a_1$. Among these $a_1$ tables,
there is an interesting relation. Indeed, let $p$ and $p'$ be two
integers such that $2 \leq p \leq s$ and $p' =s +2-p$. Then, by
using Proposition \ref{prop:Betti number}, one can show that
\begin{equation}\label{eq:duality}
\beta_{i,1}(E(r,s ,p))=\beta_{r-1-i,2}(E(r,s ,p'))\quad \mbox{and}
\quad \beta_{i,2}(E(r,s,p))=\beta_{r-1-i,1}(E(r,s ,p' )).
\end{equation}
For example, consider the case $a_1 =3$. Then $\beta(E(r,3,4))$
comes from Corollary \ref{cor:t=s+1}. Also Proposition
\ref{prop:Betti number} shows that $\beta(E(r,3,2))$ is of the form
\begin{center}
$\beta (E(r,3,2)) =$
\begin{tabular}{|c||c|c|c|c|c|c|c|}\hline
$\beta_{i,2}$ &$\beta_{0,2}$ & $\beta_{1,2}$ & $\cdots$ &
$\beta_{i,2}$ & $\cdots$ & $\beta_{r-1,2}$ & $\beta_{r,2}$  \\\hline
$\beta_{i,1}$ &$\beta_{0,1}$ & $\beta_{1,1}$ & $\cdots$ &
$\beta_{i,1}$ & $\cdots$ & $\beta_{r-1,1}$ & $\beta_{r,1}$  \\\hline
\end{tabular}
\end{center}
where
\begin{equation*}
\beta_{i,1} = \beta_{i,1} (E(r,3,2))= 2{{r-2} \choose {i}}+{{r-3}
\choose {i-1}} \quad \mbox{and} \quad \beta_{i,2} =\beta_{i,2}
(E(r,3,2))= {{r-3} \choose {i-2}}.
\end{equation*}
Now, one can quickly obtain $\beta (E(r,3,3))$ by applying
(\ref{eq:duality}) to $\beta (E(r,3,2))$. Namely, we have
\begin{center}$\beta (E(r,3,3)) =$
\begin{tabular}{|c||c|c|c|c|c|c|c|}\hline
$\beta_{i,2}$ &$\beta_{0,2}$ & $\beta_{1,2}$ & $\cdots$ &
$\beta_{i,2}$ & $\cdots$ & $\beta_{r-1,2}$ & $\beta_{r,2}$  \\\hline
$\beta_{i,1}$ &$\beta_{0,1}$ & $\beta_{1,1}$ & $\cdots$ &
$\beta_{i,1}$ & $\cdots$ & $\beta_{r-1,1}$ & $\beta_{r,1}$  \\\hline
\end{tabular}
\end{center}
where
\begin{equation*}
\beta_{i,1} = \beta_{i,1} (E(r,3,3))= {{r-3} \choose {i}} \quad
\mbox{and} \quad \beta_{i,2} =\beta_{i,2} (E(r,3,3))= 2{{r-2}
\choose {i-1}} + {{r-3} \choose {i-1}}.
\end{equation*}
\end{remark}

\noindent {\bf Proof of Proposition \ref{prop:Betti number}.}
The first part comes from Lemma \ref{lem:basics of E(r,s,t)}.(3).

For the remaining cases, note that $\beta \left(
E(r,s,t) \right) = \beta \left( E(r,s,p) \right) [\ell]$ (cf. Lemma \ref{lem:basics of E(r,s,t)}.(1)). Thus we
consider the module $E(r,s,p)$ associated to the line bundle
$\mathcal{L} := \mathcal{O}_{\P^1} (-p)$ on $S(s)$. To determine
$\beta_{i,j} \left( E(r,s,p) \right) $ for $j=1$ and $j=2$, we use
the Koszul cohomology exact sequence
\begin{equation*}
0 \rightarrow \mbox{Tor}^R _{i} (E(r,s,p),\Bbbk )_{i+j} \rightarrow
H^1 (\P^r,\bigwedge^{i+1} \mathcal{M} \otimes \mathcal{L} (j-1) )
\rightarrow \bigwedge^{i+1} V \otimes H^1 (\P^r , \mathcal{L} (j-1))
\end{equation*}
\begin{equation*}
\quad \quad \quad \rightarrow H^1 (\P^r ,\bigwedge^{i} \mathcal{M}
\otimes \mathcal{L}  (j) ) \rightarrow H^2 (\P^r ,\bigwedge^{i+1}
\mathcal{M} \otimes \mathcal{L} (j-1) ) \rightarrow \cdots
\end{equation*}
where $\mathcal{M}=\Omega _{\P^r }(1)$ and $V=H^0 (\P^r
,\mathcal{O}_{\P^r }(1))$ (cf. \cite[Theorem (1.b.4)]{G} or
\cite[Theorem 5.8]{E}). Note that the restriction of $\mathcal{M}$
to $S(s) \cong \P^1$ is isomorphic to
$\mathcal{O}_{\P^1}(-1)^{\oplus s }\oplus \mathcal{O}_{\P^1}^{\oplus
(r-s)}$.

For $j=1$, we get the cohomology vanishing $H^2 (\P^r
,\bigwedge^{i+1} \mathcal{M} \otimes \mathcal{L}  )=0$ since
$\mathcal{L}$ is supported on the curve $S(s)$. Therefore it holds
that
\begin{equation*}
\begin{split}
\beta_{i,1} (E(r,s,p)) & = {\rm dim}_{\Bbbk} ~{\rm Tor}^R_{i}(E(r,s,p),\Bbbk)_{i+1} \\
                & = h^1\big( \P^1 ,\overset{i+1}{\bigwedge} \left( \mathcal{O}_{\P^1}(-1)^{\oplus s}\oplus \mathcal{O}_{\P^1}^{\oplus (r-s)}  \right) (-p) \big)
                 - {{r+1}\choose{i+1}}h^1 \big( \P^1,\mathcal{O}_{\P^1}(-p) \big) \\
                & \quad \quad + h^1 \big( \P^1,\overset{i}{\bigwedge} \left( \mathcal{O}_{\P^1}(-1)^{\oplus s}\oplus \mathcal{O}_{\P^1}^{\oplus (r-s)}  \right)
                 \otimes \mathcal{O}_{\P^1}(s-p) \big) \\
                &=  \sum_{k=0} ^{i+1} (k+p-1){{s}\choose{k}}{{r-s}\choose{i+1-k}}- (p-1) {{r+1}\choose{i+1}} \\
                & \quad \quad + \sum_{k=s+2-p} ^i (k+p-s-1) {{s}\choose{k}}{{r-s}\choose{i-k}}\\
               \end{split}
\end{equation*}
Also one can check that
\begin{equation*}
\begin{split}
\sum_{k=0} ^{i+1} (k+p-1){{s}\choose{k}}{{r-s}\choose{i+1-k}} & =
\sum_{k=0} ^{i+1}  k {{s}\choose{k}}{{r-s}\choose{i+1-k}}
+ (p-1) \sum_{k=0} ^{i+1} {{s}\choose{k}}{{r-s}\choose{i+1-k}}  \\
                & = s {{r-1} \choose {i}} + (p-1) {{r} \choose {i+1}}
               \end{split}
\end{equation*}
and
\begin{equation*}
\sum_{k=s +2-p} ^i (k+p-s -1) {{s}\choose{k}}{{r-s}\choose{i-k}}
\quad \quad \quad \quad \quad \quad \quad \quad \quad \quad \quad
\quad \quad \quad \quad \quad \quad \quad \quad \quad \quad \quad
\end{equation*}
\begin{equation*}
\begin{split}
 & = - \sum_{k=0} ^{s+1-p} (k+p-s-1) {{s}\choose{k}}{{r-s}\choose{i-k}} + \sum_{k=0} ^i (k+p-s-1) {{s}\choose{k}}{{r-s}\choose{i-k}}  \\
                & = \sum_{k=0} ^{s+1-p} (s+1-p-k) {{s}\choose{k}}{{r-s}\choose{i-k}} + \sum_{k=0} ^i k {{s}\choose{k}}{{r-s}\choose{i-k}} \\
                & \quad \quad \quad \quad \quad \quad \quad \quad \quad \quad \quad \quad \quad \quad \quad \quad + (p-s-1)\sum_{k=0} ^i
                {{s}\choose{k}}{{r-s}\choose{i-k}}  \\
                & = \sum_{k=0} ^{s+1-p} (s+1-p-k) {{s}\choose{k}}{{r-s}\choose{i-k}} + s{{r-1} \choose {i-1}} + (p-s-1) {{r} \choose {i}}.
               \end{split}
\end{equation*}
Thus we get the desired formula for $\beta_{i,1} \left( E(r,s,p)
\right)$.

For $j=2$, we have $H^1 (\P^r , \mathcal{L} \otimes
\mathcal{O}_{\P^r} (1))=H^1 (\P^1 , \mathcal{O}_{\P^1} (s-p))=0$ and
hence
\begin{equation*}
\begin{split}
\beta_{i,2} (E(r,s,p)) & ={\rm dim}_{\Bbbk} ~{\rm Tor}^R_{i}(E(r,s,p) ,\Bbbk)_{i+2} \\
& = h^1\big( \P^1 ,\overset{i+1}{\bigwedge} \left(
\mathcal{O}_{\P^1}(-1)^{\oplus
s}\oplus \mathcal{O}_{\P^1}^{\oplus (r-s )}  \right) (s -p) \big) \\
&= \sum_{k=s +2-p} ^{i+1} (k+p-s
-1){{s}\choose{k}}{{r-s}\choose{i+1-k}}.
\end{split}
\end{equation*}
This completes the proof of the formula for $\beta_{i,2} \left(
E(r,s,p) \right)$.  \qed \\

\section{Computation of $\beta (X)$ for some cases}
\noindent In this section, we apply Theorem \ref{Thm:Decomposition 2} and Propositions \ref{prop:Betti number}
to the cases where $S$ is equal to $S(1,r-2)$ for some $r \geq 3$,
$S(2,r-3)$ for some $r \geq 6$ and $S(c,c)$ for some $c \geq 1$. As a consequence, we solve Problem ($\dagger$) when $a_1 =1$ and when $a_1 =2$ and $a_2 \geq 3$.

When $a_1 =1$, Theorem \ref{Thm:Decomposition 2} implies the
following

\begin{theorem}\label{Thm:a1=1}
Let $S$ be the smooth rational normal surface scroll  $S(1,r-2)$ in $\P^r$ and $X$ be an effective divisor of $S$ linearly equivalent to $aH+bF$ where either $a=0$ and $b \geq r-1$ or else $a
\geq 1$ and $b \geq 2$. Then
\begin{equation*}
\beta (X)= \beta (S) + \beta \left( E(r-1,r-1,1+\epsilon)
\right)[a+\delta-1] + \sum_{\ell=1} ^{\delta} \beta \left( E(r,1,2)
\right )[a+b-(r-1) (\ell-1)].
\end{equation*}
\end{theorem}

\begin{proof}
Observe that Theorem \ref{Thm:Decomposition 2} is applicable to every $X$ if $a_1 =1$ and $q_{\ell}$ is equal to $a+b-(r-3) (\ell-1)$. This completes the proof.
\end{proof}

\begin{remark}\label{rmk:S(1,r-2)}
(1) In Theorem \ref{Thm:a1=1}, the Betti tables $\beta (S)$, $\beta \left(
E(r-1,r-1,1+\epsilon) \right)$ and $\beta \left( E(r,1,2) \right )$ are completely calculated in Corollary \ref{cor:SMD}, Corollary \ref{cor:ACM divisor} and Corollary \ref{cor:t=s+1}.
\smallskip

\noindent (2) By applying Theorem \ref{Thm:a1=1}, we can reprove Theorem \ref{thm:GM}.
\smallskip

\noindent (3) When $r \geq 4$, the integers $q_1 , q_2 , \ldots ,
q_{\delta}$ decrease strictly. Thus, Theorem \ref{Thm:a1=1} shows that for fixed $\delta$ and
$a+b$, $\beta (X)$ depends only on the table $\beta \left( E(r-1,r-1,1+\epsilon) \right)$. Therefore
there are $(r-2)$ different types of $\beta (X)$ since
$\epsilon$ can take $(r-2)$ different values.
\end{remark}

\noindent {\bf Proof of Theorem \ref{thm:S(1,2)}.} By Theorem
\ref{Thm:a1=1}, we have
\begin{equation*}
\beta (X)= \beta (S) + \beta \left( E(3,3,1+\epsilon)
\right)[a+\delta-1] + \sum_{\ell=1} ^{\delta} \beta \left( E(4,1,2)
\right )[a+b-\ell-1].
\end{equation*}
Note that $\epsilon = 2$ if $b=2 \delta$ and $\epsilon =3$ if $b=2
\delta+1$. Thus the proof is completed by Remark
\ref{rmk:S(1,r-2)}.(1).                       \qed \\

\noindent {\bf Proof of Theorem \ref{thm:S(1,3)}.} By Theorem
\ref{Thm:a1=1}, we have
\begin{equation*}
\beta (X)= \beta (S) + \beta \left( E(4,4,1+\epsilon)
\right)[a+\delta-1] + \sum_{\ell=1} ^{\delta} \beta \left( E(5,1,2)
\right ) [a+b-2\ell].
\end{equation*}
Note that $\epsilon = 2$ if $b=3 \delta - 1$, $\epsilon = 3$ if $b=3
\delta$ and  $\epsilon = 4$ if $b=3 \delta + 1$. Thus the proof is
again completed by Remark \ref{rmk:S(1,r-2)}.(1).                       \qed \\

Next, we consider the case where $S$ is equal to $S(2,r-3)$ for some
$r \geq 6$. Theorem \ref{Thm:Decomposition 2} gives us the following

\begin{theorem}\label{thm:a2is2}
Let $S$ be the smooth rational normal surface scroll  $S(2,r-3)$ in $\P^r$ for some $r \geq 6$ and $X$ be an effective divisor of $S$ linearly equivalent to $aH+bF$ where either $a=0$ and $b \geq r-2$ or else $a
\geq 1$ and $b \geq 2$. Then
\renewcommand{\descriptionlabel}[1]%
             {\hspace{\labelsep}\textrm{#1}}
\begin{description}
\setlength{\labelwidth}{13mm} \setlength{\labelsep}{1.5mm}
\setlength{\itemindent}{0mm}

\item[{\rm (a)}] If $\epsilon =2$, then $\beta (X)$ is decomposed
as
\begin{equation*}
\beta (X)= \beta (S) +   \beta \left( E(H+2F ) \right) [a+\delta-2]
+ \sum_{\ell=1} ^{\delta -1} \beta \left( E(r,2, 2a+b+(5-r)(\ell-1))
\right )
\end{equation*}
where $\beta \left( E(H+2F) \right)$ is of the form
\begin{center}$\beta \left( E(H+ 2F ) \right)=$
\begin{tabular}{|c||c|c|c|c|c|}\hline
$\beta_{i,3}$ &$\beta_{0,3}$ & $\beta_{1,3}$  & $\cdots$ &
$\beta_{r-1,3}$ & $\beta_{r,3}$  \\\hline $\beta_{i,2}$
&$\beta_{0,2}$ & $\beta_{1,2}$ &   $\cdots$ & $\beta_{r-1,2}$ &
$\beta_{r,2}$  \\\hline
\end{tabular}
\end{center}
and
\begin{equation*}
\beta_{i,2}  = \begin{cases} {{r-2} \choose {i-2}}-{{r+1}\choose{i+1}}+(i+2){{r-1}\choose{i+1}} & \mbox{for $0 \leq i \leq r-4$,}\\
                              0 & \mbox{for $r-3 \leq i \leq r-1$,}
                              \end{cases}
\end{equation*}
and
\begin{equation*}
\beta_{i,3}  = \begin{cases} {{r-2} \choose {i-1}} & \mbox{for $0 \leq i \leq r-5$,}\\
                              {{i+3} \choose {r-1-i}} & \mbox{for $r-4 \leq i \leq r-3$,}\\
                              {{r+1} \choose {i+2}} & \mbox{for $r-2 \leq i \leq r-1$.}
                              \end{cases}
\end{equation*}

\item[{\rm (b)}] If $3 \leq \epsilon \leq r-2$, then $\beta (X)$ is
decomposed as
\begin{equation*}
\beta (X)= \beta (S) + \beta \left( E(r-1,r-1 ,2 + \epsilon ) \right
)[a+\delta-1] + \sum_{\ell=1} ^{\delta} \beta \left(
E(r,2,2a+b+(5-r)(\ell-1)) \right ).
\end{equation*}
\end{description}
\end{theorem}

\begin{proof}
One can check that $q_{\ell} = 2a+b + (5 - r) (\ell-1)$. Thus the two
decomposition formulas of $\beta (X)$ come immediately from Theorem \ref{Thm:Decomposition 1} and Theorem \ref{Thm:Decomposition 2}, respectively. Thus it remains to show that $\beta \left( E(H+2F) \right)$ is equal
to the one described above.

The line bundle $\mathcal{O}_S (H+2F)$ is very ample and hence there
is a smooth irreducible curve $\mathcal{C}$ on $S$ which is linearly
equivalent to $H+2F$. First we recall a geometric description of
$\mathcal{C} \subset \P^r$ (cf. \cite[Theorem 1.1]{P1}). Since $\mathcal{C}$ is contained in
$S=S(2,r-3)$ and ${\rm deg}(\mathcal{C})=r+1$, it holds that
$\mathcal{C}= \pi_P (\widetilde{\mathcal{C}})$ where
$\widetilde{\mathcal{C}} \subset \P^{r+1}$ is a rational normal
curve of degree $r+1$ and $\pi_P : \widetilde{\mathcal{C}}
\hookrightarrow \P^r$ is the isomorphic linear projection from a
point $P \in \widetilde{\mathcal{C}}^4 \setminus
\widetilde{\mathcal{C}}^3$ where $\widetilde{\mathcal{C}}^k$ is the
$k$-th join $\widetilde{\mathcal{C}}$ with itself . Thus it follows
by \cite[Theorem 1.1]{LP} that
\begin{equation}\label{eq:second Betti number}
\beta_{i,3}(\mathcal{C}) = \begin{cases} {{r-2} \choose {i-1}} & \mbox{for $0 \leq i \leq r-5$,}\\
                              {{i+3} \choose {r-1-i}} & \mbox{for $r-4 \leq i \leq r-3$, and}\\
                              {{r+1} \choose {i+2}} & \mbox{for $r-2 \leq i \leq r-1$.}
                              \end{cases}
\end{equation}
Also it holds by  \cite[Theorem 2]{Hoa} that
\begin{equation}\label{eq:Hoa}
\beta_{i,2}(\mathcal{C}) = \begin{cases} {{r} \choose {2}}-2 & \mbox{for $i=0$,}\\
\beta_{i-1,3}(\mathcal{C})+ r{{r-1}\choose{i+1}}-{{r-1}\choose{i+2}}-  {{r+1}\choose{i+1}} & \mbox{for $1 \leq i \leq r-2$, and}\\
                              0 & \mbox{for $i = r-1$.}
                              \end{cases}
\end{equation}
By \cite[Proposition 3.2]{P2}, we have $\beta \left( E(H+2F) \right)
= \beta (\mathcal{C})-\beta(S)$. Therefore we get the desired
description of $\beta \left( E(H+2F) \right)$ by combining Corollary
\ref{cor:SMD}, (\ref{eq:second Betti number}) and (\ref{eq:Hoa}).
\end{proof}

\begin{example}\label{ex:S(2,3)}
Let $X$ be an effective divisor of $S=S(2,3)$ in $\P^6$ which is
linearly equivalent to $aH+bF$ where either $a=0$ and $b \geq 4$ or
else $a \geq 1$ and $b \geq 2$. Recall that
\begin{equation*}
\delta = \left \lceil \frac{b-1}{3} \right \rceil \quad \mbox{and}
\quad q_{\ell} = 2a+b+1-\ell \quad \mbox{for all} \quad 1 \leq \ell
\leq \delta.
\end{equation*}
Write $b=6m+k$ for some $2 \leq k \leq 7$. For the simplicity, we
denote $\beta \left( E(6,2,2)\right )[a+2m+t]$ and $\beta \left(
E(6,2,3)\right )[a+2m+t]$ by $T_2(t)$ and $T_3(t)$, respectively. By
Theorem \ref{thm:a2is2}, $\beta (X)$ can be decomposed into exactly
one of the following six types according to the value of
$b~(\mbox{mod}~6)$.

\renewcommand{\descriptionlabel}[1]%
             {\hspace{\labelsep}\textrm{#1}}
\begin{description}
\setlength{\labelwidth}{13mm} \setlength{\labelsep}{1.5mm}
\setlength{\itemindent}{0mm}

\item[{\rm \underline{\textit{Case 1.}}}] If $b=6m+2$, then $\epsilon =2$ and $\beta (X)$ is decomposed as
\begin{equation*}
\beta (X)= \beta (S) +   \beta \left( E(H+2F ) \right) [a+2m-1]+ T_3(0)+ \sum_{k=1} ^{m-1}\{ T_2(k)+T_3(k)\} + T_2(m)
\end{equation*}
\item[{\rm \underline{\textit{Case 2.}}}] If $b=6m+3$, then $\epsilon  =3$ and $\beta (X)$ is decomposed as
\begin{equation*}
\beta (X)= \beta (S) +   \beta \left( E(5,5,5) \right) [a+2m]
+ T_3(0) + \sum_{k=1} ^{m}\{ T_2(k)+T_3(k)\}
\end{equation*}
\item[{\rm \underline{\textit{Case 3.}}}] If $b=6m+4$, then $\epsilon  =4$ and $\beta (X)$ is decomposed as
\begin{equation*}
\beta (X)= \beta (S) +   \beta \left( E(5,5,6) \right) [a+2m]
+ \sum_{k=1} ^{m}\{ T_2(k)+T_3(k)\} + T_2(m+1)
\end{equation*}
\item[{\rm \underline{\textit{Case 4.}}}] If $b=6m+5$, then $\epsilon =2$ and $\beta (X)$ is decomposed as
\begin{equation*}
\beta (X)= \beta (S) +   \beta \left( E(H+2F ) \right) [a+2m]
+ T_3(1) + \sum_{k=1} ^{m}\{ T_2(k+1)+T_3(k+1) \}
\end{equation*}

\item[{\rm \underline{\textit{Case 5.}}}] If $b=6m+6$, then $\epsilon  =3$ and $\beta (X)$ is decomposed as
\begin{equation*}
\beta (X)= \beta (S) +   \beta \left( E(5,5,5) \right) [a+2m+1]
+ T_3(1) + \sum_{k=1} ^{m}\{T_2(k+1)+T_3(k+1) \} + T_2(m+2)
\end{equation*}

\item[{\rm \underline{\textit{Case 6.}}}] If $b=6m+7$, then $\epsilon  =4$ and $\beta (X)$ is decomposed as
\begin{equation*}
\beta (X)= \beta (S) +   \beta \left( E(5,5,6 ) \right) [a+2m+1]
+ \sum_{k=1} ^{m+1}\{ T_2(k+1)+T_3(k+1)\}.
\end{equation*}
\end{description}
Also, using Corollary \ref{cor:SMD}, Corollary \ref{cor:ACM
divisor}, Example \ref{ex:E(r,2,p)} and Theorem \ref{thm:a2is2}.(a),
we can obtain $\beta (S)$, $\beta \left(E(H+2F) \right )$, $\beta
\left( E(5,5,5)  \right )$, $\beta \left( E(5,5,6)  \right )$,
$\beta \left( E(6,6,2)  \right )$ and $\beta \left( E(6,6,3)  \right
)$. The precise form of $\beta (X)$ for \textit{Case 1} -
\textit{Case 6} are provided respectively in the following Table 3 and
Table 4.

\begin{table}[hbt]
\begin{center}
\begin{tabular}{|c||c|c|c|c|c|c||c|c|c|c|c|c||c|c|c|c|c|c|}\hline
$i$ & $1$      & $2$ & $3$      & $4$   & $5$  & $6$ & $1$      &
$2$ & $3$      & $4$   & $5$  & $6$ & $1$      & $2$      & $3$ &
$4$ & $5$  & $6$ \\\hline\hline $\beta_{i,a+3m+3}$ & $0$ & $0$ & $0$
& $0$   & $0$  & $0$ & $0$      & $0$      & $0$ & $0$ & $0$  & $0$
& $0$      & $1$      & $4$      & $6$ & $4$  & $1$
\\\hline $\beta_{i,a+3m+2}$ & $0$      & $1$      & $4$      & $6$
& $4$  & $1$ & $2$      & $11$      & $24$      & $26$   & $14$  &
$3$ & $3$      & $15$      & $30$      & $30$  & $15$  & $3$
\\\hline $\beta_{i,a+3m+1}$ & $3$      & $15$      & $30$      &
$30$  & $15$  & $3$   & $3$      & $15$      & $30$      & $30$  &
$15$  & $3$& $3$      & $15$      & $30$      & $30$  & $15$  & $3$
\\\hline $\vdots$               & $\vdots$ & $\vdots$ & $\vdots$ &
$\vdots$ & $\vdots$ & $\vdots$ & $\vdots$ & $\vdots$ & $\vdots$ &
$\vdots$ & $\vdots$ & $\vdots$  & $\vdots$ & $\vdots$ & $\vdots$ &
$\vdots$ & $\vdots$ & $\vdots$\\\hline $\beta_{i,a+2m+3}$ & $3$
& $15$      & $30$      & $30$  & $15$  & $3$  & $3$      & $15$
& $30$      & $30$  & $15$  & $3$   & $3$      & $15$      & $30$ &
$30$  & $15$  & $3$  \\\hline $\beta_{i,a+2m+2}$ & $3$      & $15$
& $36$      & $39$ & $18$   & $3$  & $3$      & $24$      & $46$
& $39$ & $15$   & $2$ & $6$      & $24$      & $36$      & $24$ &
$6$   & $0$ \\\hline $\beta_{i,a+2m+1}$   & $3$      & $9$      &
$6$      & $0$ & $0$  & $0$   & $1$      & $0$      & $0$      & $0$
& $0$  & $0$   & $0$      & $0$      & $0$      & $0$ & $0$  & $0$
\\\hline $\beta_{i,a+2m}$          & $0$      & $0$      & $0$
& $0$  & $0$   & $0$  & $0$      & $0$      & $0$      & $0$  & $0$
& $0$   & $0$      & $0$      & $0$      & $0$  & $0$   &
$0$\\\hline $\vdots$               & $\vdots$ & $\vdots$ & $\vdots$
& $\vdots$ & $\vdots$& $\vdots$ & $\vdots$ & $\vdots$ & $\vdots$ &
$\vdots$ & $\vdots$& $\vdots$ & $\vdots$ & $\vdots$ & $\vdots$ &
$\vdots$ & $\vdots$& $\vdots$  \\\hline $\beta_{i,3}$          & $0$
& $0$      & $0$      & $0$  & $0$   & $0$  & $0$      & $0$      &
$0$      & $0$  & $0$   & $0$ & $0$      & $0$      & $0$      & $0$
& $0$   & $0$  \\\hline $\beta_{i,2}$          & $10$      & $20$
& $15$      & $4$  & $0$   & $0$  & $10$      & $20$      & $15$
& $4$  & $0$   & $0$ & $10$      & $20$      & $15$      & $4$  &
$0$   & $0$\\\hline
\end{tabular}
\end{center}
\caption{$b=6m+2$, $b=6m+3$ and $b=6m+4$}
\end{table}

\begin{table}[hbt]
\begin{center}
\begin{tabular}{|c||c|c|c|c|c|c||c|c|c|c|c|c||c|c|c|c|c|c|}\hline
$i$ & $1$      & $2$ & $3$      & $4$   & $5$  & $6$ & $1$      &
$2$ & $3$      & $4$   & $5$  & $6$ & $1$      & $2$      & $3$ &
$4$   & $5$  & $6$ \\\hline\hline  $\beta_{i,a+3m+4}$ & $0$      &
$0$      & $0$      & $0$   & $0$  & $0$ & $0$      & $1$      & $4$
& $6$   & $4$  & $1$ & $2$ & $11$      & $24$      & $26$   & $14$ &
$3$ \\\hline $\beta_{i,a+3m+3}$ & $2$      & $11$      & $24$ & $26$
& $14$  & $3$ & $3$      & $15$      & $30$      & $30$  & $15$  &
$3$ & $3$      & $15$      & $30$      & $30$  & $15$  & $3$
\\\hline $\beta_{i,a+3m+2}$ & $3$      & $15$      & $30$      &
$30$  & $15$ & $3$  & $3$      & $15$      & $30$      & $30$  &
$15$  & $3$ & $3$      & $15$      & $30$      & $30$  & $15$  & $3$
\\\hline $\vdots$               & $\vdots$ & $\vdots$ & $\vdots$ &
$\vdots$ & $\vdots$ & $\vdots$& $\vdots$ & $\vdots$ & $\vdots$ &
$\vdots$ & $\vdots$ & $\vdots$& $\vdots$ & $\vdots$ & $\vdots$ &
$\vdots$ & $\vdots$ & $\vdots$\\\hline $\beta_{i,a+2m+4}$ & $3$
& $15$ & $30$      & $30$  & $15$  & $3$ & $3$      & $15$      &
$30$ & $30$  & $15$  & $3$  & $3$      & $15$      & $30$ & $30$  &
$15$ & $3$  \\\hline $\beta_{i,a+2m+3}$ & $3$      & $15$      &
$36$ & $39$ & $18$   & $3$  & $3$      & $24$      & $46$      &
$39$ & $15$   & $2$ & $6$      & $24$ & $36$      & $24$ & $6$   &
$0$
\\\hline $\beta_{i,a+2m+2}$   & $3$      & $9$      & $6$      & $0$
& $0$  & $0$ & $1$      & $0$      & $0$      & $0$ & $0$  & $0$   &
$0$      & $0$      & $0$      & $0$ & $0$  & $0$   \\\hline
$\beta_{i,a+2m+1}$          & $0$      & $0$      & $0$      & $0$
& $0$   & $0$  & $0$      & $0$      & $0$      & $0$  & $0$   & $0$
& $0$      & $0$      & $0$ & $0$ & $0$  & $0$ \\\hline $\vdots$
& $\vdots$ & $\vdots$ & $\vdots$ & $\vdots$ & $\vdots$& $\vdots$ &
$\vdots$ & $\vdots$ & $\vdots$ & $\vdots$ & $\vdots$ & $\vdots$&
$\vdots$ & $\vdots$ & $\vdots$ & $\vdots$ & $\vdots$ & $\vdots$
\\\hline $\beta_{i,3}$          & $0$      & $0$      & $0$      &
$0$  & $0$   & $0$  & $0$      & $0$      & $0$      & $0$  & $0$
& $0$  & $0$ & $0$      & $0$      & $0$  & $0$ & $0$ \\\hline
$\beta_{i,2}$          & $10$      & $20$      & $15$      & $4$  &
$0$   & $0$  & $10$      & $20$      & $15$      & $4$  & $0$   &
$0$ & $10$      & $20$      & $15$      & $4$  & $0$   & $0$
\\\hline
\end{tabular}
\end{center}
\caption{$b=6m+5$, $b=6m+6$ and $b=6m+7$}
\end{table}
Here, the Betti numbers lying in the vertical dots on Table 3 and Table 4 are as follows:\\

$\begin{array}{lll}
\begin{tabular}{|c||c|c|c|c|c|c|c|}\hline
$\beta_{i,a+j}$ & $3$      & $15$      & $30$      & $30$ & $15$
& $3$     \\\hline
\end{tabular} &\quad\text{ for $\lceil\frac{b-1}{3}\rceil+2 \leq j \leq \lceil\frac{b-1}{3}\rceil+\lceil\frac{b-7}{6}\rceil$ and}\\
\end{array}$
\smallskip

$\begin{array}{lll}
\begin{tabular}{|c||c|c|c|c|c|c|c|}\hline
$\beta_{i,k}$         & $0$      & $0$      & $0$      & $0$  & $0$
& $0$      \\\hline
\end{tabular} & \quad\quad\quad\quad\text{for $3 \leq k \leq a+\lceil\frac{b-4}{3}\rceil$}
\end{array}$\\
\end{example}

We finish this section by providing some examples which show that
the hypotheses $a_2 \geq 2 a_1 -1$ and $a_1 +1 \leq \epsilon \leq
a_2 +1$ in Theorem \ref{Thm:Decomposition 2} can not be weakened.

\begin{example}\label{example:S(2,3)}
Let $X$ be an effective divisor of $S=S(2,3)$ linearly equivalent
to $H+11F$. We have
\begin{equation*}
\delta (X)= 4, \quad \epsilon (X)= 2 \quad \mbox{and} \quad q_{\ell}
(X)= 14-\ell \quad \mbox{for $1 \leq \ell \leq 4$ }.
\end{equation*}
Theorem \ref{thm:a2is2} says that $\beta (X)$ is decomposed as
\begin{equation*}
\beta (X)=\beta (S) + \beta \left( E(H+2F ) \right) [3] + \beta
\left( E(6,2,3)\right) [4]+\beta \left( E(6,2,2)\right) [5] +\beta
\left( E(6,2,3)\right) [5].
\end{equation*}
Now, by using Corollary \ref{cor:SMD}, Example \ref{ex:E(r,2,p)} and
Theorem \ref{thm:a2is2}.(a), we can write $\beta (X)$ explicitly as
below.\\

\begin{center}
$\beta (X)=$
\begin{tabular}{|c||c|c|c|c|c|c|}\hline
$\beta_{i,7}$  & $2$ & $11$ & $24$  & $26$  & $14$  & $3$    \\\hline
$\beta_{i,6}$  & $3$ & $15$ & $36$  & $39$  & $18$  & $3$    \\\hline
$\beta_{i,5}$  & $3$ & $9$  & $6$   & $0$   & $0$   & $0$    \\\hline
$\beta_{i,4}$  & $0$ & $0$  & $0$   & $0$   & $0$   & $0$    \\\hline
$\beta_{i,3}$  & $0$ & $0$  & $0$   & $0$   & $0$   & $0$    \\\hline
$\beta_{i,2}$  & $10$& $20$ & $15$  & $4$   & $0$   & $0$    \\\hline
\end{tabular}
\end{center}
From $\beta (X)$ in Example \ref{example:S(2,3)}, we can see that
the hypothesis $a_1 + 1 \leq \epsilon (X) \leq a_2 +1$ in Theorem
\ref{Thm:Decomposition 2} cannot be weakened. Indeed, let $T$ denote
the right hand side of (\ref{eq:Decomposition 2}) when $X$ is equal
to the divisor $X$ of $S(2,3)$. That is,
\begin{equation*}
T := \beta (S) + \beta \left( E(5,5,4) \right )[4]  \quad \quad
\quad \quad \quad \quad \quad \quad \quad \quad \quad \quad \quad
\quad \quad \quad \quad \quad
\end{equation*}
\begin{equation*}
\quad \quad \quad \quad \quad \quad \quad  +\beta \left( E(6,2,2)
\right )[4]+\beta \left( E(6,2,3) \right )[4]+\beta \left( E(6,2,2)
\right )[5]+\beta \left( E(6,2,3) \right ) [5].
\end{equation*}
One can easily check that $\beta (X) \neq T$ and hence Theorem
 \ref{Thm:Decomposition 2} fails to hold for $X$.
\end{example}
\smallskip

Finally, we consider some curves on the smooth rational normal surface scroll $S=S(c,c)$ for some $c
\geq 1$.

\begin{theorem}\label{thm:S(c,c)}
Let $S=S(c ,c) \subset \P^{2c+1}$ be a smooth rational normal
surface scroll and let $X$ be an effective divisor of $S$ linearly
equivalent to $aH+ (uc +1)F$ for some $a \geq 0$ and $u \geq 1$.
Then $\beta (X)$ is of the form\\

\begin{center}
$\beta (X) =$
\begin{tabular}{|c||c|c|c|c|c|}\hline
$\beta_{i,a+u+1}$ & $\beta_{0,a+u+1}$ & $\beta_{1,a+u+1}$ & $\cdots$ & $\beta_{r-1,a+u+1}$ & $\beta_{r,a+u+1}$ \\\hline
$\beta_{i,a+u}$   & $0$               & $0$               & $\cdots$ & $0$                 & $0$                \\\hline
$\vdots$          & $\vdots$          & $\vdots$          & $\ddots$ & $\vdots$            & $\vdots$           \\\hline
$\beta_{i,3}$     & $0$               & $0$               & $\cdots$ & $0$                 & $0$                \\\hline
$\beta_{i,2}$     & $\beta_{0,2}$     & $\beta_{1,2}$     & $\cdots$ & $\beta_{r-1,2}$       & $\beta_{r,2}$    \\\hline
\end{tabular}
\end{center}
\smallskip

\noindent where
\begin{equation*}
\beta_{i,2} = (i+1) {{2c} \choose {i+2}} \quad \mbox{and} \quad
\beta_{i,,a+u+1} = (i+1) {{2c} \choose {i+1}} + uc {{2c} \choose
{i}} \quad \mbox{for all} \quad i \geq 0.
\end{equation*}
\end{theorem}

\begin{proof}
One can check that $\delta (X) = u$, $\epsilon (X) = c +1$ and
$q_{\ell} (X) = (a+u)c +1$ for all $1 \leq \ell \leq u$. Thus, by
using Proposition \ref{prop:fundamental exact sequence}.(5)
repeatedly, we obtain the decomposition
\begin{equation*}
\beta (X)= \beta (Z) + \sum_{\ell=1} ^{u} \beta \left( E(2c+1, c,
(a+u)c +1 ) \right)
\end{equation*}
of $\beta (X)$ where $Z$ is an irreducible divisor of $S$ linearly
equivalent to $(a+u)H+F$. Note that $Z \subset \P^{2c+1}$ is
arithmetically Cohen-Macaulay (cf. \cite[Theorem 4.3]{P2}). Now, let
$\Gamma \subset \P^{2c}$ be a general hyperplane section of $Z$.
Thus we have
\begin{equation*}
\beta (Z) = \beta (\Gamma) \quad \mbox{and} \quad |\Gamma| =
2c(a+u)+1.
\end{equation*}
Also $\Gamma$ is contained in $S(2c)$ since $Z$ is a divisor of $S$.
By \cite[Proposition 3.2]{P2}, it follows that
\begin{equation*}
\beta (\Gamma) = \beta (S(2c)) + \beta \left( E(2c,2c,2c(a+u)+1 )
\right).
\end{equation*}
In consequence, $\beta (X)$ is decomposed as
\begin{equation}\label{eq:decomposition S(c,c)}
\beta (X)= \beta (S) + \beta \left( E(2c,2c ,2c + 1 ) \right
)[a+u-1] + u \times \beta \left( E(2c+1, c, c+1 ) \right) [a+u-1].
\end{equation}
Thus we get the desired result by combining Corollary \ref{cor:ACM
divisor}, Corollary \ref{cor:t=s+1} and (\ref{eq:decomposition
S(c,c)}).
\end{proof}

\begin{remark}
When $S$ is the smooth quadric $S(1,1)$ in
$\P^3$, we can apply Theorem \ref{thm:S(c,c)} to every effective divisor $X$ of $S$ linearly equivalent
to $aH+bF$ for some $a \geq 0$ and $b \geq 2$. Therefore, Theorem \ref{thm:S(c,c)} reproves Theorem \ref{thm:GM}.
\end{remark}

\begin{example}\label{ex:S(2,2)}
Let $S=S(2,2)$ in $\P^5$ and $X$ be an effective divisor of $S$
linearly equivalent to $aH+bF$ where either $a=0$ and $b \geq 3$ or else $a \geq 1$ and $b \geq 2$.

\noindent (1) When $b=2m+1$ for some $m \geq 1$,
Theorem \ref{thm:S(c,c)} shows that

\begin{center}
$\beta (X) =$
\begin{tabular}{|c||c|c|c|c|c|c|}\hline
$\beta_{i,a+m+1}$ & $b+3$   & $4b+8$  & $6b+6$ & $4b$   & $\quad b-1
\quad$      \\\hline $\beta_{i,a+m}$   & $0$      & $0$ & $0$ & $0$
& $0$       \\\hline $\vdots$          & $\vdots$ & $\vdots$ &
$\vdots$ & $\vdots$ & $\vdots$   \\\hline $\beta_{i,3}$ & $0$      &
$0$      & $0$      & $0$      & $0$
\\\hline $\beta_{i,2}$     & $6$      & $8$      & $3$      & $0$
& $0$       \\\hline
\end{tabular}.
\end{center}
\smallskip

\noindent (2) Consider the case where $X\equiv H+4F$. Thus $\delta
(X) = \epsilon (X) = 2$ and $q_1 (X) = 6$. Let $T$ denote the right
hand side of (\ref{eq:Decomposition 1}) for this $X$. That is,
\begin{equation*}
T := \beta (S) +  \beta \left( E(H+2F ) \right) [1] + \beta \left(
E(5,2,2) \right )[2].
\end{equation*}
Then we have
\begin{center}
$T =$
\begin{tabular}{|c||c|c|c|c|c|}\hline
$\beta_{i,4}$ & $0$    & $5$      & $13$     & $9$      & $2$
\\\hline $\beta_{i,3}$ & $3$    & $6$      & $3$      & $1$      &
$0$  \\\hline $\beta_{i,2}$ & $6$    & $8$      & $3$      & $0$ &
$0$  \\\hline
\end{tabular}
$\quad$ and $\quad$ $\beta (X) =$
\begin{tabular}{|c||c|c|c|c|c|}\hline
$\beta_{i,4}$ & $0$    & $5$      & $12$     & $9$      & $2$
\\\hline $\beta_{i,3}$ & $3$    & $6$      & $3$      & $0$      &
$0$  \\\hline $\beta_{i,2}$ & $6$    & $8$      & $3$      & $0$ &
$0$  \\\hline
\end{tabular}.
\end{center}
where $\beta (X)$ is computed by means of the computer algebra
system SINGULAR \cite{GP}. In particular, $T \neq \beta (X)$. This example shows that the hypothesis
$a_2 \geq 2a_1 -1$ in Theorem \ref{Thm:Decomposition 1} cannot be
weakened.
\end{example}


\begin{thebibliography}{0000000}
\bibitem[EH]{EH} D. Eisenbud and J. Harris,
{\em On varieties of minimal degree (A centennial account)},
Proceedings of Symposia in Pure Mathematics 46 (1987), 3-13.

\bibitem[E]{E} D. Eisenbud, {\em The Geometry of Syzygies}, no.229, Springer-Velag New York, (2005)

\bibitem[Fe]{Fe} R. Ferraro, {\em Weil divisors on rational normal scrolls}, Lecture Notes in
Pure and Applied Mathematics, 217 (2001), 183-198.

\bibitem[GP]{GP} Gert-Martin Greuel, Gerhard Pfister et al {\em Singular 3.0, a computer algebra
system for polynomial computations}, Center for Computer Algebra,
University of Kaiserslautern (2005) (http://www.singular.uni-kl.de).

\bibitem[GM]{G-M} S. Giuffrida and R. Maggioni, {\em On the
resolution of a curve lying on a smooth cubic surface in $\P^3$},
Trans. Am. Math. Soc. 331 (1992), 181-201.

\bibitem[G]{G} M. Green, {\em Koszul cohomology and the geometry of projective varieties},
J. Differential Geom. 19 (1984), 125-171.

\bibitem[H]{H} J. Harris, {\em Curves in projective space. With the collaboration of
David Eisenbud}, Seminaire de Mathematiques Superieures, 85. Presses
de l'Universite de Montreal, Montreal, Que., (1982).

\bibitem[Hoa]{Hoa} L. T. Hoa, {\em On minimal free resolutions of projective varieties of
degree=codimension+2}, J. Pure Appl. Algebra 87 (1993), 241-250.

\bibitem[LP]{LP} W. Lee and E. Park, {\em Projective curves of degree=codimension+2 II}, Int. J. Algebra Comput. Vol. 26, No. 1, (2016), 95-104.

\bibitem[M]{M} C. Miyazaki, {\em Sharp bounds on Castelnuovo-Mumford
regularity}, Trans. Amer. Math. Soc. 352 (2000), no. 4, 1675-1686.

\bibitem[MV]{MV} C. Miyazaki and W. Vogel, {\em Bounds on cohomology and Castelnuovo-Mumford regularity}, J. Algebra 185  (1996),  no. 3, 626-642.

\bibitem[N]{N} U. Nagel, {\em Arithmetically Buchsbaum divisors on varieties of minimal degree}, Trans. Am. Math. Soc. 351, 4381-4409 (1999)

\bibitem[P1]{P1} E. Park, {\em Projective curves of degree = codimension+2}, Math. Z. 256 (2007), no. 3, 685-697.

\bibitem[P2]{P2} E. Park {\em On syzygies of divisors of rational normal scrolls}, Math. Nachr. 287 (2014), no. 11-12, 1383 - 1393.
\end{thebibliography}
\end{document}